\documentclass[reqno]{amsart}
\usepackage{amssymb}
\usepackage{upgreek}
\usepackage{xcolor}

\newtheorem{theorem}{Theorem}[section]

\newtheorem{lemma}[theorem]{Lemma}

\newtheorem{corollary}[theorem]{Corollary}
\theoremstyle{definition}

\newtheorem{remark}[theorem]{Remark}

\numberwithin{equation}{section}
\allowdisplaybreaks

\begin{document}
\title[Certain product formulas and values of hypergeometric series]
{Certain product formulas and values of Gaussian hypergeometric series}



\author{Mohit Tripathi}
\address{Department of Mathematics, Indian Institute of Technology Guwahati, North Guwahati, Guwahati-781039, Assam, INDIA}
\curraddr{}
\email{m.tripathi@iitg.ac.in}
 
\author{Rupam Barman}
\address{Department of Mathematics, Indian Institute of Technology Guwahati, North Guwahati, Guwahati-781039, Assam, INDIA}
\curraddr{}
\email{rupam@iitg.ac.in}
\thanks{We thank the anonymous referee for his/her thorough review and highly appreciate the comments and suggestions, which significantly contributed to improving the quality of the paper. The second author is partially supported by a research grant under the MATRICS scheme of SERB, Department of Science and Technology, Government of India.}


\date{Research in Number Theory, vol. 6 (2020)}
\begin{abstract}
In this article we find finite field analogues of certain product formulas satisfied by the classical hypergeometric series. 
We express product of two ${_2}F_1$-Gaussian hypergeometric series as ${_4}F_3$- and ${_3}F_2$-Gaussian hypergeometric series. 
We use properties of Gauss and Jacobi sums and our earlier works on finite field Appell series to deduce these product formulas satisfied by the Gaussian hypergeometric series. 
We then use these transformations to evaluate explicitly some special values of ${_4}F_3$- and ${_3}F_2$-Gaussian hypergeometric series. 
By counting points on CM elliptic curves over finite fields, Ono found certain special values of ${_2}F_1$- and ${_3}F_2$-Gaussian hypergeometric series containing trivial 
and quadratic characters as parameters. Later, Evans and Greene found special values of certain ${_3}F_2$-Gaussian hypergeometric series containing arbitrary characters as parameters from where 
some of the values obtained by Ono follow as special cases. We show that some of the results of Evans and Greene follow from our product formulas including a finite field analogue of the classical Clausen's identity. 
\end{abstract}
\maketitle
\section{Introduction and statement of results}
For a complex number $a$, the rising factorial is defined as $(a)_0=1$ and $(a)_k=a(a+1)\cdots (a+k-1), ~k\geq 1$.
For a non-negative integer $n$, and $a_i, b_i\in\mathbb{C}$ with $b_i\notin\{\ldots, -3,-2,-1, 0\}$,
the (generalized) hypergeometric series ${_{n+1}}F_{n}$ is defined by
\begin{align}\label{hyper}
{_{n+1}}F_{n}\left(\begin{array}{cccc}
                   a_1, & a_2, & \ldots, & a_{n+1} \\
                    & b_1, & \ldots, & b_n
                 \end{array}\mid x
\right):=\sum_{k=0}^{\infty}\frac{(a_1)_k\cdots (a_{n+1})_k}{(b_1)_k\cdots(b_n)_k}\cdot\frac{x^k}{k!},
\end{align}
which converges absolutely for $|x|<1$.
In 1980s, Greene \cite{greene-thesis, greene} introduced a finite field, character sum analogue of classical hypergeometric series that satisfies summation and transformation
properties similar to those satisfied by the classical hypergeometric series. Let $p$ be an odd prime, and let $\mathbb{F}_q$ denote the finite field with $q$ elements, where $q=p^r, r\geq 1$.
Let $\widehat{\mathbb{F}_q^{\times}}$ be the group of all multiplicative characters on $\mathbb{F}_q^{\times}$. We extend the domain of each $\chi\in \widehat{\mathbb{F}_q^{\times}}$ to $\mathbb{F}_q$ by 
setting $\chi(0)=0$ including the trivial character $\varepsilon$. For multiplicative characters $A$ and $B$ on $\mathbb{F}_q$, the binomial coefficient ${A \choose B}$ is defined by
\begin{align}\label{eq-0}
{A \choose B}:=\frac{B(-1)}{q}J(A,\overline{B})=\frac{B(-1)}{q}\sum_{x \in \mathbb{F}_q}A(x)\overline{B}(1-x),
\end{align}
where $J(A, B)$ denotes the usual Jacobi sum and $\overline{B}$ is the character inverse of $B$. 
For positive integer $n$, Greene \cite{greene} defined the ${_{n+1}}F_n$-hypergeometric series over
$\mathbb{F}_q$ by
\begin{align}\label{Greene-def-4}
{_{n+1}}F_n\left(\begin{array}{cccc}
                A_0, & A_1, & \ldots, & A_n\\
                 & B_1, & \ldots, & B_n
              \end{array}\mid x \right)
              =\frac{q}{q-1}\sum_{\chi\in \widehat{\mathbb{F}_q^\times}}{A_0\chi \choose \chi}{A_1\chi \choose B_1\chi}
              \cdots {A_n\chi \choose B_n\chi}\chi(x),
\end{align}
where $A_0, A_1,\ldots, A_n$ and $B_1, B_2,\ldots, B_n$ are multiplicative characters on $\mathbb{F}_q$. Hypergeometric series over finite fields are also known as \textit{Gaussian} hypergeometric series.
\par There are other finite field analogues of the classical hypergeometric series. For example, see \cite{katz, mccarthy3, FL}. For a multiplicative character $\chi$, 
let $g(\chi)$ denote the Gauss sum as defined in Section 2. 
For $A_0, A_1, \ldots A_n, B_1, B_2, \ldots, B_n\in\widehat{\mathbb{F}_q^{\times}}$, the McCarthy's finite field hypergeometric function ${_{n+1}F}^{\ast}_n$ is given by
\begin{align}\label{mcCarty's_defn}
&{_{n+1}F}_n\left(\begin{array}{cccc}
                A_0, & A_1, & \ldots, & A_n\\
                 & B_1, & \ldots, & B_n
              \end{array}\mid x \right)^{\ast}\notag\\
              &=\frac{1}{q-1}\sum_{\chi\in \widehat{\mathbb{F}_q^\times}}\prod_{i=0}^{n}\frac{g(A_i\chi)}{g(A_i)}\prod_{j=1}^n\frac{g(\overline{B_j\chi})}{g(\overline{B_j})}g(\overline{\chi})
\chi(-1)^{n+1}\chi(x).
\end{align}
In \cite[Proposition 2.5]{mccarthy3}, McCarthy proved that his finite field hypergeometric series is closely related to Greene's 
hypergeometric series. To be specific, let $A_0\neq\varepsilon$ and $A_i\neq B_i$ for $1\leq i\leq n$. Then for $x\in \mathbb{F}_q$ we have
\begin{align}\label{prop-300}
&{_{n+1}F}_n\left(\begin{array}{cccc}
                A_0, & A_1, & \ldots, & A_n\\
                 & B_1, & \ldots, & B_n
              \end{array}\mid x \right)^{\ast}\notag\\
              &=\left[\prod_{i=1}^n{A_i \choose B_i}^{-1}\right]{_{n+1}F}_n\left(\begin{array}{cccc}
                A_0, & A_1, & \ldots, & A_n\\
                 & B_1, & \ldots, & B_n
              \end{array}\mid x \right).
              \end{align}

\par In a recent paper \cite{FL}, Fuselier et. al. introduce another version of hypergeometric series over finite fields in a manner that is parallel to that of the classical hypergeometric series
by considering period functions for hypergeometric type algebraic varieties over finite fields. For multiplicative characters $A, B, C$, their ${_2\mathbb{F}_1}$-hypergeometric series is given by 
\begin{align}\label{Fuselier-def}
{_{2}}\mathbb{F}_1\left[\begin{array}{cccc}
A, & B \\
& C
\end{array}\mid x \right]:= \frac{1}{J(B,\overline{B}C)}{_{2}}\mathbb{P}_1\left[\begin{array}{cccc}
A, & B \\
& C
\end{array}\mid x \right],
\end{align}
where
\begin{align}
{_{2}}\mathbb{P}_1\left[\begin{array}{cccc}
A, & B \\
& C
\end{array}\mid x \right]:=\frac{q^2}{(q-1)}BC(-1)\sum_{\chi\in \widehat{\mathbb{F}_q^\times}}{A\chi\choose\chi}{B\chi\choose C\chi}\chi(x) +\delta(x)J(B,\overline{B}C).\notag
\end{align}
Here $\delta$ denotes the function defined on $\mathbb{F}_q$ by $\delta(0)=1$ and $\delta(x)=0$ if $x\neq 0$.
The relationship between the above finite field  hypergeometric series and the Greene's hypergeometric series  is the following:
\begin{align}\label{relation-Fu-Greene}
{_{2}}\mathbb{F}_1\left[\begin{array}{cccc}
A, & B \\
& C
\end{array}\mid x \right]=\frac{qBC(-1)}{J(B,\overline{B}C)}{_{2}}F_1\left(\begin{array}{cccc}
A, & B \\
& C
\end{array}\mid x \right) + \delta(x).
\end{align}
We note that, since we have used the definition of the binomial coefficient given by Greene, the above definition of ${_{2}}\mathbb{P}_1$ series differs from its original definition given in \cite{FL} 
by a factor of $q^2$.

\par Throughout this paper, $A, B, C, D, E, F, S, \chi, \lambda, \psi$ denote multiplicative characters on $\mathbb{F}_q$. Here $\varepsilon$ and $\varphi$ always
denote the trivial and quadratic characters, respectively, while $\chi_4$ denotes a fixed quartic character when $q\equiv 1\pmod{4}$. Also, $\chi_3$ denotes a fixed cubic character when $q\equiv 1\pmod{3}$.
For brevity, if $A$ is a square we write $A=\square$. 
\subsection{Product formulas for Gaussian hypergeometric series}
Greene \cite{greene} found several transformation formulas satisfied by the Gaussian hypergeometric series analogous to those 
satisfied by the classical hypergeometric series. Since then many mathematicians have obtained finite field analogues of transformation and summation identities satisfied by the classical
hypergeometric series (see for example \cite{Evans-2, EG, EG-2, EG-3, FL, mccarthy3, TB}). 
Finite field hypergeometric series are known to be related to various arithmetic objects. Some of the biggest motivations for studying finite field hypergeometric functions
have been their connections with Fourier coefficients and eigenvalues of modular forms and with counting points on certain kinds of algebraic varieties. Assuming the conjecture of van Geemen and van Straten, McCarthy and Papanicolas \cite{MM} related the eigenvalue of the Hecke 
operator of index $p$ of a Siegel eigenform of degree $2$ and level $8$ to ${_{4}}F_3\left(\begin{array}{cccc}
                \varphi, & \varphi, & \varphi, & \varphi\\
                 & \varepsilon, & \varepsilon, & \varepsilon
              \end{array}\mid -1 \right)$. The following identity played a crucial role in their proof:
              \begin{align*}
               {_{4}}F_3\left(\begin{array}{cccc}
                \varphi, & \varphi, & \varphi, & \varphi\\
                 & \varepsilon, & \varepsilon, & \varepsilon
              \end{array}\mid -1 \right)={_{2}}F_1\left(\begin{array}{cc}
                \varphi, & \varphi\\
                 & \varepsilon
              \end{array}\mid -1 \right) \cdot {_{3}}F_2\left(\begin{array}{ccc}
                \chi_4, & \varphi, & \varphi\\
                 & \varepsilon, & \varepsilon
              \end{array}\mid 1 \right).
              \end{align*}
In \cite{EG, EG-1}, Evans and Greene expressed ${_3}F_2$-hypergeometric series as a product of ${_2}F_1$-hypergeometric series over finite fields from where they deduced certain special values 
of ${_3}F_2$-hypergeometric series including a finite field analogue of the Clausen's identity. In this paper, we prove finite field analogues of certain product formulas satisfied by the 
classical hypergeometric series. In the following theorem, we express a ${_4}F_3$-hypergeometric series as a product of two ${_2}F_1$-hypergeometric series over finite fields. 
\begin{theorem}\label{MT41}
	Let $A, B, C \in \widehat{\mathbb{F}_q^{\times}}$ be such that $A^2, B^2\neq\varepsilon$, $A^2\neq C$, and $B^2\neq C$. For $x\neq 1$, we have
	\begin{align}
		&{_{2}}F_1\left(\begin{array}{cccc}
			A^2, & B^2 \\
			& C
		\end{array}\mid x \right){_{2}}F_1\left(\begin{array}{cccc}
			A^2, & B^2 \\
			& A^2B^2\overline{C}\end{array}\mid x \right)\notag \\
		&= \frac{qAB(4)g(\overline{A^2})g(AB\overline{C})g(\overline{AB}C\varphi)}{g(\overline{B^2})g(B^2\overline{C})g(\overline{A^2}C)g(\varphi)} {_{4}}F_3\left(\begin{array}{cccccc}
			\hspace{-.15cm} A^2, \hspace{-.15cm} & B^2, \hspace{-.15cm}& AB, \hspace{-.15cm}& AB\varphi \\
			& A^2B^2, & C, & A^2B^2\overline{C} 
		\end{array}\mid 4x(1-x) \right)\notag\\
		&-\frac{(q-1)AB(4)g(\overline{A^2})g(AB\overline{C})g(\overline{AB}C\varphi)}{g(\overline{B^2})g(B^2\overline{C})g(\overline{A^2}C)g(\varphi)}\left[{_{3}}F_2\left(\begin{array}{cccccc}
			\hspace{-.15cm} A^2, \hspace{-.15cm} & B^2,  & \hspace{-.2cm}AB\varphi \\
			& A^2B^2, & \hspace{-.2cm} A^2B^2\overline{C} 
		\end{array}\hspace{-.1cm} \mid 4x(1-x) \right)\right.\notag\\
		&\left. \times \delta(AB\overline{C})+{{_3}}F_2\left(\begin{array}{cccccc}
			A^2, & B^2, & AB \\
			& A^2B^2, & C 
		\end{array}\mid 4x(1-x) \right)\delta(\overline{AB}C\varphi) \right]\notag\\
		&+\frac{q\overline{C}\overline{A^2}(1-x)C\overline{B^2}(x)}{g(A^2)g(\overline{B^2})g(B^2\overline{C})g(\overline{A^2}C)}\delta\left(\frac{1-2x}{(1-x)^2}\right)-\frac{(q-1)g(A\overline{B})g(\overline{A}B)\overline{AB}(x-x^2)}{q^2g(A^2)g(\overline{B^2})g(B^2\overline{C})g(\overline{A^2}C)}
		\notag\\
		&\times [(q-1)\delta(AB)\delta(AB\overline{C})- qAB(-1)\delta(AB\overline{C})-qABC(-1)\delta(AB)] \notag\\
		&- \frac{(q-1)g(\overline{A}B\varphi)g(A\overline{B}\varphi)\overline{AB}\varphi(x-x^2)}{q^2g(A^2)g(\overline{B^2})g(B^2\overline{C})g(\overline{A^2}C)}
		[(q-1)\delta(AB\overline{C}\varphi)\delta(AB\varphi)\notag\\
		&-qAB\varphi(-1)\delta(AB\overline{C}\varphi) -q ABC\varphi(-1)\delta(AB\varphi) ].\notag
	\end{align}
\end{theorem}
We show that many interesting results proved by Evans, Greene, and Ono follow from the above transformation including a finite field analogue of the Clausen's classical identity. 
We have stated Theorem \ref{MT41} with minimum conditions on the parameters so that certain known results can be deduced, and therefore 
there are some extra terms in the formula. The extra terms will disappear if we put some additional conditions on the parameters. For example, we have the following corollary.
\begin{corollary}\label{revision-2}
	Let $A, B, C \in \widehat{\mathbb{F}_q^{\times}}$ be such that $A^2, B^2, A^2B^2, A^2B^2\overline{C^2}\neq\varepsilon$, $A^2\neq C$, and $B^2\neq C$. 
	For $x \neq 1,\frac{1}{2}$, we have
	\begin{align}
		&{_{2}}F_1\left(\begin{array}{cccc}
			A^2, & B^2 \\
			& C
		\end{array}\mid x \right){_{2}}F_1\left(\begin{array}{cccc}
			A^2, & B^2 \\
			& A^2B^2\overline{C}\end{array}\mid x \right)\notag \\
		&= \frac{qAB(4)g(\overline{A^2})g(AB\overline{C})g(\overline{AB}C\varphi)}{g(\overline{B^2})g(B^2\overline{C})g(\overline{A^2}C)g(\varphi)} {_{4}}F_3\left(\begin{array}{cccccc}
			\hspace{-.15cm}A^2, & \hspace{-.15cm}B^2, & \hspace{-.15cm}AB, & \hspace{-.15cm}AB\varphi \\
			& \hspace{-.15cm}A^2B^2, & \hspace{-.15cm} C, & \hspace{-.15cm} A^2B^2\overline{C} 
		\end{array}\mid 4x(1-x) \right).\notag
	\end{align}
\end{corollary}                                        
If we apply \eqref{prop-300} to Corollary \ref{revision-2}, we obtain the following identity satisfied by the McCarthy's finite field hypergeometric series.
\begin{align*}
&{_{2}}F_1\left(\begin{array}{cccc}
A^2, & B^2 \\
& C
\end{array}\mid x \right)^{\ast}{_{2}}F_1\left(\begin{array}{cccc}
A^2, & B^2 \\
& A^2B^2\overline{C}\end{array}\mid x \right)^{\ast}\\
&= {_{4}}F_3\left(\begin{array}{cccccc}
A^2, & B^2, & AB, & AB\varphi \\
& A^2B^2, & C, & A^2B^2\overline{C} 
\end{array}\mid 4x(1-x) \right)^{\ast}.
\end{align*}

The above identity is a finite field analogue of the following identity \cite[(6.1)]{bailey-3} satisfied by the classical hypergeometric series:
\begin{align}
	&{_{2}}F_1\left(\begin{array}{ccc}
		\alpha, & \beta\\
		& \gamma
	\end{array}\mid x\right){_{2}}F_1\left(\begin{array}{ccc}
		\alpha, & \beta\\
		& \alpha + \beta -\gamma
	\end{array}\mid x\right)\notag\\
	&={_{4}}F_3\left(\begin{array}{ccccc}
		\alpha, & \beta, & \frac{1}{2}(\alpha+\beta), & \frac{1}{2}(\alpha+\beta+1)\\
		& \alpha+\beta, & \gamma, & \alpha+\beta-\gamma+1
	\end{array}\mid 4x(1-x)\right).\notag
\end{align}

The following transformation satisfied by the classical hypergeometric series is equivalent to the Clausen's identity \cite{bailey-3}.
\begin{align}\label{identity-3f2-1}
&	{_{2}}F_1\left(\begin{array}{ccc}
		\alpha, & \beta\\
		& \alpha+\beta+\frac{1}{2}
	\end{array}\mid 4x(1-x)\right)^2\notag\\
	&={_{3}}F_2\left(\begin{array}{cccc}
		2\alpha, & 2\beta, & \alpha+\beta\\
		& 2\alpha+ 2\beta, & \alpha+\beta+\frac{1}{2}
	\end{array}\mid 4x(1-x)\right).
\end{align}
From Theorem \ref{MT41}, we prove the following result which is a finite field analogue of \eqref{identity-3f2-1}.
\begin{theorem}\label{MT41C2}
	Let $A, B  \in \widehat{\mathbb{F}_q^{\times}}$ be such that $A^2, B^2, A\overline{B}\varphi, AB, AB\varphi\neq\varepsilon$. 
	For $x \neq 1, \frac{1}{2}$, we have
	\begin{align}
		&{_{2}}F_1\left(\begin{array}{cccc}
			A, & B \\
			& AB\varphi
		\end{array}\mid 4x(1-x) \right)^2 = \frac{AB(4)g(B)^2g(A\varphi)^2}{qg(A^2)g(B^2)}\notag\\
		&\hspace{1cm} \times 
		{_{3}}F_2\left(\begin{array}{cccccc}
			A^2, & B^2, & AB \\
			& A^2B^2, & AB\varphi 
		\end{array}\mid 4x(1-x) \right)+\frac{g(B)^2g(A\varphi)^2\overline{AB}\varphi(x-x^2)}{q^2g(A^2)g(B^2)}.\notag
	\end{align}                        
\end{theorem}  
We note that a finite field analogue of the Clausen's identity was also obtained by Evans and Greene \cite[Thm 1.5]{EG}. Theorem \ref{MT41C2} can also be deduced from \cite[Thm 1.5]{EG} by taking $S=B, C=AB\varphi$, and then employing Lemma \ref{gj1} and Lemma \ref{g1}.
\par The following identity expresses a ${_4}F_3$ classical hypergeometric series as a product of two ${_2}F_1$ classical hypergeometric series \cite[(7.4)]{bailey-3}.
\begin{align}\label{identity-4f3-1}
	&{_{2}}F_1\left(\begin{array}{ccc}
		\alpha, & \beta\\
		& \gamma
	\end{array}\mid x\right){_{2}}F_1\left(\begin{array}{ccc}
		\gamma-\beta, & 1-\beta\\
		& \alpha-\beta-1
	\end{array}\mid x\right)\notag\\
	&=(1-x)^{\beta-\alpha-\gamma}{_{4}}F_3\left(\begin{array}{ccccc}
		\hspace{-.15cm} \alpha, &\hspace{-.2cm} \gamma- \beta, & \hspace{-.15cm} \frac{1}{2}(\alpha+\gamma-\beta), & \hspace{-.2cm} \frac{1}{2}(\alpha+\gamma-\beta+1)\\
		& \hspace{-.2cm} \alpha+\gamma-\beta, & \hspace{-.2cm} \gamma, & \hspace{-.15cm} \alpha-\beta+1
	\end{array}\hspace{-.15cm} \mid \frac{-4x}{(1-x)^2}\right).
\end{align}
In the following theorem, we prove a finite field analogue of \eqref{identity-4f3-1}.
\begin{theorem}\label{MT42}
	Let $A, D, E\in\widehat{\mathbb{F}_q^{\times}}$ be such that $A^2, E^2, A^2\overline{D^2E^2}, A^2D^2\overline{E^2}\neq\varepsilon$, $A^2\neq D^2$, and $D^2\neq E^2$. 
	For $z\neq 1$, we have
	\begin{align}
		&{_{2}}F_1\left(\begin{array}{cccc}
			A^2, & E^2 \\
			& D^2
		\end{array}\mid z \right){_{2}}F_1\left(\begin{array}{cccc}
			D^2\overline{E^2}, & \overline{E^2} \\
			& A^2\overline{E^2}\end{array}\mid z \right)\notag\\
		&= \frac{E^2(z)}{q}\delta\left(1-z^2\right)+ \frac{AD\overline{E}(4)\overline{A^2D^2}E^2(1-z)g(A\overline{ED})g(\overline{A}ED\varphi)}{g(\varphi)}\notag\\
		&\hspace{3cm}\times {_{4}}F_3\left(\begin{array}{cccccc}
			A^2, & D^2\overline{E^2}, & AD\overline{E}, & AD\overline{E}\varphi \\
			& A^2D^2\overline{E^2}, & D^2, & A^2\overline{E^2} 
		\end{array}\mid \frac{-4z}{(1-z)^2} \right).\notag                                             
	\end{align}
\end{theorem}
If we assume $z^2\neq 1$ in Theorem \ref{MT42}, then \eqref{prop-300} yields
\begin{align}
&{_{2}}F_1\left(\begin{array}{cccc}
A^2, & E^2 \\
& D^2
\end{array}\mid z \right)^{\ast}{_{2}}F_1\left(\begin{array}{cccc}
D^2\overline{E^2}, & \overline{E^2} \\
& A^2\overline{E^2}\end{array}\mid z \right)^{\ast}\notag\\
&= \overline{A^2D^2}E^2(1-z){_{4}}F_3\left(\begin{array}{cccccc}
A^2, & D^2\overline{E^2}, & AD\overline{E}, & AD\overline{E}\varphi \\
& A^2D^2\overline{E^2}, & D^2, & A^2\overline{E^2} 
\end{array}\mid \frac{-4z}{(1-z)^2} \right)^{\ast},\notag                                             
\end{align}
which is an exact finite field analogue of \eqref{identity-4f3-1}.
\par The following is another product formula satisfied by the classical hypergeometric series \cite[(6.3)]{bailey-3}.
\begin{align}\label{identity-4f3-2}
	&{_{2}}F_1\left(\begin{array}{ccc}
		\alpha, & \beta\\
		& \gamma
	\end{array}\mid x\right){_{2}}F_1\left(\begin{array}{ccc}
		\alpha, & \gamma-\beta\\
		& \gamma
	\end{array}\mid x\right)\notag\\
	&=(1-x)^{-\alpha}{_{4}}F_3\left(\begin{array}{ccccc}
		\alpha, & \beta, & \gamma-\alpha, & \gamma-\beta\\
		& \gamma, & \frac{1}{2}\gamma, & \frac{\gamma+1}{2}
	\end{array}\mid \frac{-x^2}{4(1-x)}\right).
\end{align}
We prove the following result which is a finite field analogue of \eqref{identity-4f3-2}.
\begin{theorem}\label{MT43}
	Let $A, B, C\in\widehat{\mathbb{F}_q^{\times}}$ be such that $A, B, C^2\neq\varepsilon$ and $A, B\neq C^2$. For $x \neq 1$, we have
	\begin{align}
		&{_{2}}F_1\left(\begin{array}{cccc}
			\hspace{-.15cm} A, & \hspace{-.15cm} B \\
			&\hspace{-.15cm} C^2
		\end{array}\mid x \right){_{2}}F_1\left(\begin{array}{cccc}
			\hspace{-.15cm} A, & \hspace{-.15cm} C^2\overline{B} \\
			& \hspace{-.15cm} C^2
		\end{array}\mid x \right)
		= \frac{qAB(-1)\overline{A^2}B(1-x)\overline{C^2}(x)}{g(A)g(\overline{B})g(\overline{A}C^2)g(B\overline{C^2})}\delta\left(\frac{x-2}{x-1}\right)\notag\\
		&+\frac{q\overline{A}(1-x)g(\overline{A}C)g(\overline{B}C\varphi)}
		{\varphi(-1)C(4)g(\varphi)g(\overline{A}C^2)g(\overline{B})}{_{4}}F_3\left(\begin{array}{cccccc}
			A, & B, & \overline{A}C^2, & \overline{B}C^2 \\
			& C^2, & C, & C\varphi 
		\end{array}\mid \frac{-x^2}{4(1-x)}\right)\notag\\
		&+\frac{(q-1)\overline{A}(1-x)g(\overline{A}C)g(\overline{B}C\varphi)}{\varphi(-1)C(4)g(\varphi)g(\overline{A}C^2)g(\overline{B})}\left[\frac{q-1}{q} {_{2}}F_1\left(\begin{array}{ccccccc}
			\hspace{-.2cm} A, & \hspace{-.2cm} B \\
			& \hspace{-.2cm}C^2
		\end{array}\mid \frac{-x^2}{4(1-x)} \right)\delta(\overline{A}C)\delta(\overline{B}C\varphi)\right.\notag\\
		&\left.-{_{3}}F_2\left(\begin{array}{ccccccc}
			\hspace{-.2cm} A, & \hspace{-.2cm} B, &\hspace{-.2cm}  \overline{B}C^2 \\
			& \hspace{-.2cm} C^2, & \hspace{-.2cm} C\varphi\end{array} \hspace{-.2cm} \mid \frac{-x^2}{4(1-x)} \right)\delta(\overline{A}C) 
		- {_{3}}F_2\left(\begin{array}{ccccccc}
			\hspace{-.2cm} A, & \hspace{-.2cm} B, & \hspace{-.2cm} \overline{A}C^2 \\
			& \hspace{-.2cm} C^2, & \hspace{-.2cm} C
		\end{array} \hspace{-.2cm}\mid \frac{-x^2}{4(1-x)} \right)\delta(\overline{B}C\varphi)\right]\notag\\
		&-\frac{(q-1)\overline{A}(1-x)\overline{C}(x^2)C(1-x)}{q g(A)g(\overline{B})g(B\overline{C^2})g(\overline{A}C^2)}[(q-1)\delta(A\overline{C}) \delta(B\overline{C})-qBC(-1)\delta(A\overline{C})\notag\\
		&-qAC(-1)\delta(B\overline{C})+(q-1)\varphi(1-x)\delta(A\overline{C}\varphi)\delta(B\overline{C}\varphi)-qBC(-1)\varphi(x-1)\delta(A\overline{C}\varphi)\notag\\
		&-qAC(-1)\varphi(x-1)\delta(B\overline{C}\varphi)].\notag                                            
	\end{align}
\end{theorem}

If we put some additional conditions on the parameters in Theorem \ref{MT43}, we readily obtain the following identity.
\begin{corollary}\label{MT43C}
	Let $A, B, C\in\widehat{\mathbb{F}_q^{\times}}$ be such that $A, B, C^2, A^2\overline{C^2}, B^2\overline{C^2}\neq\varepsilon$ and 
	$A, B\neq C^2$. For $x\neq 1$, we have
		\begin{align}
		&{_{2}}F_1\left(\begin{array}{cccc}
		\hspace{-.1cm}	A, & \hspace{-.2cm} B \\
			& \hspace{-.2cm} C^2
		\end{array}\mid x \right){_{2}}F_1\left(\begin{array}{cccc}
			\hspace{-.1cm} A, & \hspace{-.2cm} C^2\overline{B} \\
			& \hspace{-.2cm} C^2
		\end{array}\mid x \right)
		= \frac{qAB(-1)\overline{A^2}B(1-x)\overline{C^2}(x)}{g(A)g(\overline{B})g(\overline{A}C^2)g(B\overline{C^2})}\delta\left(\frac{x-2}{x-1}\right)\notag\\
		&\hspace{.4cm}+\frac{q\varphi(-1)\overline{C}(4)\overline{A}(1-x)g(\overline{A}C)g(\overline{B}C\varphi)}{g(\varphi)g(\overline{A}C^2)g(\overline{B})}{_{4}}F_3\left(\begin{array}{cccccc}
			\hspace{-.1cm} A, & \hspace{-.2cm} B, & \hspace{-.2cm} \overline{A}C^2, &\hspace{-.2cm} \overline{B}C^2 \\
			& \hspace{-.2cm} C^2, & \hspace{-.2cm} C, & \hspace{-.2cm} C\varphi 
		\end{array}\mid \frac{-x^2}{4(1-x)}\right).\notag
	\end{align}
\end{corollary}
If we assume $x\neq 2$ in Corollary \ref{MT43C}, then \eqref{prop-300} yields
\begin{align}
&{_{2}}F_1\left(\begin{array}{cccc}
A, & B \\
& C^2
\end{array}\mid x \right)^{\ast}{_{2}}F_1\left(\begin{array}{cccc}
A, & C^2\overline{B} \\
& C^2
\end{array}\mid x \right)^{\ast}
\notag\\
&= \overline{A}(1-x){_{4}}F_3\left(\begin{array}{cccccc}
A, & B, & \overline{A}C^2, & \overline{B}C^2 \\
& C^2, & C, & C\varphi 
\end{array}\mid \frac{-x^2}{4(1-x)}\right)^{\ast},\notag
\end{align}
which is an exact finite field analogue of \eqref{identity-4f3-2}.
\subsection{Special values of Gaussian hypergeometric series} Finding special values of Gaussian hypergeometric series is an important and
interesting problem. Special values of Gaussian hypergeometric series play an important role in solving
many old conjectures and supercongruences. Many special values of ${_{2}}F_1$- and ${_{3}}F_2$-Gaussian hypergeometric series are obtained by using different techniques 
(see for example \cite{Ahlgren, BK, BK-1, EG-1, greene, GS, Kalita, ono, sadek, TB}). 
Finding values of Gaussian hypergeometric series containing arbitrary characters at specific values of the argument 
is a difficult problem. In this article, we have used our product formulas to find special values of ${_{4}}F_3$- and ${_{3}}F_2$-hypergeometric series. 
In the following theorem, we find special values of ${_{4}}F_3$-hypergeometric series at general values of the argument.
\begin{theorem}\label{Value-41}
	Let $q\equiv 1\pmod{4}$. Let $A \in \widehat{\mathbb{F}_q^{\times}}$ be such that  $A^2\not\in \{\varepsilon, \varphi, \chi_4, \overline{\chi_4}\}$. 
	For $x\neq 0, 1$, we have
	\begin{align}
		(i)~&{_{4}}F_3\left(\begin{array}{cccccc}
			A^2, & A^2\varphi, & A^2\chi_4, & A^2\overline{\chi_4} \\
			& A^4\varphi, & A^4, & \varphi 
		\end{array}\mid 4x(1-x)\right)
		=\frac{\overline{A^4}\varphi(2)}{g(\varphi)g(A^2\chi_4)g(\overline{A^2}\chi_4)}\notag\\
		&\times \left(\frac{1+\varphi(1-x)}{2}\right)\left(\frac{1+\varphi(x)}{2}\right)
		\left(\overline{A^4}(1+\sqrt{1-x})+ \overline{A^4}(1-\sqrt{1-x})\right)\notag\\
		&\times \left(\overline{A^4}(1+\sqrt{x})+ \overline{A^4}(1-\sqrt{x})\right)-\frac{A^2\varphi(x)\overline{A^4}\varphi(2)\overline{A^6}(x-1)g(\varphi)}{qg(A^2\chi_4)g(\overline{A^2}\chi_4)}\delta\left(\frac{1-2x}{(1-x)^2}\right),\notag\\
		(ii)~&{_{4}}F_3\left(\begin{array}{cccccc}
			A^2, & A^2\varphi, & A^2\chi_4, & A^2\overline{\chi_4} \\
			& A^4\varphi, & A^4, & \varphi 
		\end{array}\mid \frac{-4x}{(1-x)^2}\right)
		=\frac{\overline{A^4}\varphi(2)}{g(\varphi)g(A^2\chi_4)g(\overline{A^2}\chi_4)}\notag\\
		&\times \frac{(1+\varphi(1-x))(1+\varphi(x^2-x))}{4}
		\left(\overline{A^4}\left(1+\frac{1}{\sqrt{1-x}}\right)+ \overline{A^4}\left(1-\frac{1}{\sqrt{1-x}}\right)\right)\notag\\
		&\times\left(\overline{A^4}\left(1+\sqrt{\frac{x}{x-1}}\right)+ \overline{A^4}\left(1-\sqrt{\frac{x}{x-1}}\right)\right)\notag\\
		&-\frac{A^2\varphi(x)\overline{A^4}\varphi(2)A^4\varphi(x-1)g(\varphi)}{qg(A^2\chi_4)g(\overline{A^2}\chi_4)}\delta(1-x^2).\notag
	\end{align}
\end{theorem}
We note that the above formulas are well-defined. Since $q\equiv 1\pmod{4}$, $x-1$ is a square if and only if $1-x$ is a square. 
In (i), if $x$ or $1-x$ is not a square, then the term containing the product $(1+\varphi(x))(1+\varphi(1-x))$ will disappear. In (ii), 
if $x$ or $1-x$ is not a square, then the term containing the product $(1+\varphi(x^2-x))(1+\varphi(1-x))$ will disappear. 
\par Putting $x=\frac{1}{2}$ in Theorem \ref{Value-41} (i) we find the following special value of a ${_{4}}F_3$-Gaussian hypergeometric series.
\begin{corollary}\label{V41C1}
	Let $q\equiv 1\pmod{4}$. Let $A \in \widehat{\mathbb{F}_q^{\times}}$ be such that  $A^2\not\in \{\varepsilon, \varphi, \chi_4, \overline{\chi_4}\}$. We have
	\begin{align}
		&{_{4}}F_3\left(\begin{array}{cccccc}
			A^2, & A^2\varphi, & A^2\chi_4, & A^2\overline{\chi_4} \\
			& A^4\varphi, & A^4, & \varphi 
		\end{array}\mid 1 \right)
		= \displaystyle - \frac{g(\varphi)}{qg(A^2\chi_4)g(\overline{A^2}\chi_4)}\notag\\
		&+\left\{
		\begin{array}{ll}
		\displaystyle\frac{1}{g(\varphi)g(A^2\chi_4)g(\overline{A^2}\chi_4)}
		\left[2 + \overline{A^8}\left(1+\sqrt{2}\right)+ \overline{A^8}\left(1-\sqrt{2}\right)\right], 
		&\hspace{-.2cm} \hbox{if $q\equiv 1 \hspace{-.2cm}\pmod{8}$;} \\
		0, & \hspace{-.2cm} \hbox{if $q\equiv 5\hspace{-.2cm}\pmod{8}$.}
		\end{array}\right.	\notag
	\end{align}	
\end{corollary}
In \cite{ono}, Ono found several special values of ${_{2}}F_1$- and ${_{3}}F_2$-Gaussian hypergeometric series containing trivial and quadratic characters 
as parameters by counting points on CM elliptic curves. 
We find the following special value which generalizes a result of Ono.
\begin{theorem}\label{Value-44}
 Let $A\in \widehat{\mathbb{F}_q^{\times}}$ be such that $A^2, A^6\neq\varepsilon$. Then we have
\begin{align}
 &{_{3}}F_2\left(\begin{array}{ccccccc}
             A^2, & A^6, & A^4\varphi \\
                  & A^8, & A^4 
        \end{array}\mid -8\right) = \frac{\overline{A}(256)g(A^2)^2 g(\overline{A^6})}{qg(\overline{A^2})}\left[{A^3\choose A^2} + {A^3\varphi\choose A^2}\right]^2\notag\\
&\hspace{4cm}-\frac{\overline{A}(4096)}{q}-\frac{q-1}{q^3}\overline{A}(4096)\varphi(2)g(\overline{A^2}\varphi)g(A^2\varphi)\delta(A^4\varphi).\notag       
\end{align}
\end{theorem}
Putting $A=\chi_4 $ in Theorem \ref{Value-44} we readily obtain the following special value obtained by Ono \cite{ono} when $q\equiv 1 \pmod{4}$.
\begin{corollary}\label{ono-8} Let $q\equiv 1\pmod{4}$. We have
\begin{align}
	{_{3}}F_2\left(\begin{array}{ccccccc}
		\varphi, & \varphi, & \varphi \\
		& \varepsilon, & \varepsilon
	\end{array}\mid -8\right)=	\left[{\chi_4\choose\varphi}+ {\overline{\chi_4}\choose\varphi}\right]^2 - \frac{1}{q}.\notag
\end{align}
\end{corollary}

Using our product formulas, we next find special values of ${_3}F_2$-hypergeometric series at $x=-1$ and $x=\frac{1}{4}$, respectively. 
We note that these two results were also proved by Evans and Greene, see for example \cite[Thm 1.3 \& 1.4]{EG-1}. 
\begin{theorem}\label{Value-45}
 Suppose that $C$ is a multiplicative character whose order is not equal to $1, 2, 4$. Then for $q\equiv 1\pmod 8$ we have
\begin{align}
	&{_{3}}F_2\left(\begin{array}{ccccccc}
		\varphi, & C^2\varphi, & C\varphi \\
		& C^2, & C 
	\end{array}\mid -1\right)= \left\{
	\begin{array}{ll}
		\displaystyle\frac{1}{q}, & \hbox{if $C\chi_{4}\neq\square$ ;} \vspace{.14cm}\\
		\displaystyle\frac{1}{q} + \frac{2}{q^2}Re(J(D, \varphi)J(\overline{D}\chi_{4}, \varphi)), & \hbox{if $C\chi_{4} = D^{2}$.}
	\end{array}
	\right.\notag   
\end{align}
 \end{theorem}
 \begin{theorem}\label{Value-46}
Suppose that $C$ is a multiplicative character which is a square and its order is strictly greater than $4$. Then we have
  \begin{align}
 & {_{3}}F_2\left(\begin{array}{ccccccc}
 \hspace{-.1cm}	\overline{C}, &\hspace{-.2cm} C^3, & \hspace{-.2cm} C \\
 	&\hspace{-.2cm} C^2, &\hspace{-.2cm} C\varphi 
 \end{array}\hspace{-.2cm} \mid \frac{1}{4}\right)= \left\{ \hspace{-.2cm}
 \begin{array}{ll}
 	\displaystyle -\frac{C(4)}{q}, & \hspace{-.25cm} \hbox{if $q\equiv 11 \hspace{-.25cm}\pmod{12}$;} \vspace{.1cm}\\
 	\displaystyle\frac{C(4)}{q}\left[q+ 2Re(J(C, \chi_{3})J(\overline{C}, \chi_{3}))\right], &\hspace{-.25cm} \hbox{if $q\equiv  1 \hspace{-.25cm} \pmod{12}.$}
 \end{array}
 \right.\notag  
 \end{align}
 \end{theorem}
In the following theorem, we find values of ${_3}F_2$-hypergeometric series at $x=-8$. 
 \begin{theorem}\label{Value-43}
Let $A \in \widehat{\mathbb{F}_q^{\times}}$ be such that $A^2, A^6\neq\varepsilon$ and $A^4\neq\varphi$. Then we have
 \begin{align}
 &{_{3}}F_2\left(\begin{array}{ccccccc}
               \overline{A^2}, & A^2 , & \varphi \\
                  & A^4, & \overline{A^4} 
         \end{array}\mid -8\right)
 = \frac{g(\varphi)}{g(\overline{A^4})g(A^4\varphi)}{\overline{A^2}\choose\overline{A^4}}^{-1}\left[{\overline{A}\choose A^2} + {\varphi\overline{A}\choose A^2} \right]\notag\\
 &\times \left[{A\choose \overline{A^2}}+ {\varphi A\choose \overline{A^2}}\right]+\frac{q-1}{q}{\overline{A^2}\choose\overline{A^4}}^{-1}{_{3}}F_2\left(\begin{array}{ccccccc}
              A^2 , & \overline{A^2} , & \varphi \\
                  & \varepsilon, & \overline{A^4} 
         \end{array}\mid -8\right)\delta(\overline{A^4})\notag\\
         &-\frac{1}{q^2}{\varphi\choose\overline{A^4}}{\overline{A^2}\choose\overline{A^4}}^{-1}
 -\frac{(q-1)g(\varphi)}{q^3g(\overline{A^4})g(A^4\varphi)}{\overline{A^2}\choose\overline{A^4}}^{-1}\left[ \delta(\overline{A^4}) + q \right].\notag 
 \end{align}
\end{theorem}
We remark that Corollary \ref{ono-8} also follows from Theorem \ref{Value-43} by taking $A=\chi_4$. 
We also show that the following result of Evans and Greene \cite[Thm 1.9]{EG} 
follows from Theorem \ref{Value-43}.
\begin{theorem}\label{MTV43C1}
Suppose that $A$ is a multiplicative character whose order is not equal to $1, 2, 3, 4, 6, 8$. Then we have
\begin{align}
&{_{3}}F_2\left(\begin{array}{ccccccc}
\varphi, &A^2, & \overline{A^2} \\
& A^4, & \overline{A^4}
\end{array}\mid -8\right)
= \frac{1}{q} + \frac{\overline{A^2}(4)J(\overline{A^2}, A^6)}{q^2J(A^2, A^2)}\left[J(A^2, A)^2 + J(A^2, A\varphi)^2\right].\notag
\end{align}
\end{theorem}
In the following theorem we find values of ${_3}F_2$-hypergeometric series at $x=4$. 
\begin{theorem}\label{Value-49}
Suppose that $S$ is a multiplicative character which is a square and its order is strictly greater than $4$. Then we have
\begin{align}
	&{_{3}}F_2\left(\begin{array}{ccccccc}
	\overline{S^3}, & \overline{S}, &  \overline{S^2}\varphi \\
	& \overline{S^4}, & \overline{S^2} 
\end{array}\mid 4\right)=-\frac{\varphi(-1)S(16)}{q}\notag\\
&+ \frac{S(16)\overline{S}(27)J(\overline{S}, \overline{S})}{J(\overline{S^3}, S)}\left\{
\begin{array}{ll}
	0, & \hbox{if $q \equiv 11 \mod(12) $;} \\
	\displaystyle  \left[{S\choose \chi_3} + {S \choose \chi_3^2}\right]^2, & \hbox{if $q \equiv 1 \mod(12)$.}
\end{array}\right.\notag
\end{align} 
\end{theorem}	
We note that the above result was also proved by Kalita and the second author by counting points on certain algebraic curves over finite fields, 
see for example \cite[Thm 1.7]{BK}. Here we present a different proof using our product formulas.
\section{Notation and Preliminaries} We first recall some definitions and results from \cite{greene}. 
Let $\delta$ denote the function on multiplicative characters defined by
$$\delta(A)=\left\{
              \begin{array}{ll}
                1, & \hbox{if $A$ is the trivial character;} \\
                0, & \hbox{otherwise.}
              \end{array}
            \right.
$$
We also denote by $\delta$ the function defined on $\mathbb{F}_q$ by 
$$\delta(x)=\left\{
              \begin{array}{ll}
                1, & \hbox{if $x=0$;} \\
                0, & \hbox{if $x\neq 0$.}
              \end{array}
            \right.
$$
The binomial coefficient ${A\choose B}$ defined in \eqref{eq-0} satisfies many interesting properties. For example, we list the following from \cite{greene}:
\begin{align}\label{b5}
{A\choose \varepsilon}={A\choose A}=\frac{-1}{q}+\frac{q-1}{q}\delta(A);
\end{align}

\begin{align}\label{b7}
{\varepsilon\choose A}=-\frac{A(-1)}{q}+\frac{q-1}{q}\delta(A);
\end{align}
and
\begin{align}\label{b6}
    {A\choose B}{C\choose A}={C\choose B}{C\overline{B}\choose A\overline{B}}-\frac{q-1}{q^2}B(-1)\delta(A)
    +\frac{q-1}{q^2}AB(-1)\delta(B\overline{C}).
\end{align}
\par We next recall some properties of Gauss and Jacobi sums. For further details, see \cite{evans}. Let $\zeta_p$ be a fixed primitive $p$-th root of unity
in ${\mathbb{C}}$. The trace map $\text{tr}: \mathbb{F}_q \rightarrow \mathbb{F}_p$ is given by
\begin{align}
\text{tr}(\alpha)=\alpha + \alpha^p + \alpha^{p^2}+ \cdots + \alpha^{p^{r-1}}.\notag
\end{align}
Then the additive character
$\theta: \mathbb{F}_q \rightarrow \mathbb{C}$ is defined by
\begin{align}
\theta(\alpha)=\zeta_p^{\text{tr}(\alpha)}.\notag
\end{align}
For $\chi \in \widehat{\mathbb{F}_q^\times}$, the \emph{Gauss sum} is defined by
\begin{align}
g(\chi):=\sum\limits_{x\in \mathbb{F}_q}\chi(x)\theta(x).\notag
\end{align}
We let $T$ denote a fixed generator of $\widehat{\mathbb{F}_q^\times}$. 
\begin{lemma}\emph{(\cite[(1.12)]{greene})}\label{g1}
If $k\in\mathbb{Z}$, then
$$g(T^k)g(T^{-k})=qT^k(-1)-(q-1)\delta(T^k).$$
\end{lemma}
\begin{lemma}\emph{(\cite[(17)]{FL})}\label{g3} For $A\in\widehat{\mathbb{F}_q^{\times}}$ we have
\begin{align}
\frac{1}{g(\overline{A})}= \frac{A(-1)g(A)}{q}-\frac{(q-1)}{q}\delta(A).\notag
\end{align}
\end{lemma}
\begin{lemma}\emph{(\cite[(1.14)]{greene})}\label{gj1} For $A, B\in\widehat{\mathbb{F}_q^{\times}}$ we have
\begin{align}
J(A, B)=\frac{g(A)g(B)}{g(AB)}+(q-1)B(-1)\delta(AB).\notag
\end{align}
\end{lemma}
Using Lemma \ref{gj1}, we can re-write the binomial coefficient in terms of Gauss sums as follows.
\begin{lemma}\emph{\label{g8}}
If $A, B \in \widehat{\mathbb{F}_q^{\times}}$ then we have
 \begin{align}
     {A \choose B}=\frac{B(-1)g(A)g(\overline{B})}{qg(A\overline{B})}+ \frac{q-1}{q}\delta(A\overline{B}). \notag
 \end{align}
\end{lemma}
The orthogonality relation for multiplicative characters is given in the following lemma.
\begin{lemma}\emph{(\cite[(1.16)]{greene})}\label{g5}
If $x\in\mathbb{F}_q^{\times}$ then we have
\begin{align}
 \sum_{\chi\in \widehat{\mathbb{F}_q^{\times}}}\chi(x) = (q-1)\delta(1-x).\notag
\end{align}
\end{lemma}
Another important product formula for Gauss sums is the Davenport-Hasse relation.
\begin{lemma}[Davenport-Hasse relation \cite{Lang}]\emph{\label{HD}} Let $\chi$ be a character of order $m$ on  $\mathbb{F}_q^{\times}$, for some positive integer $m$. For character $A$ on $\mathbb{F}_q^{\times}$ we have
	\begin{align}
		\prod_{i=0}^{m-1}g(\chi^{i}A) = g(A^m)\overline{A^{m}}(m)\prod_{i=1}^{m-1}g(\chi^{i}).\notag
	\end{align}
\end{lemma}
We use Davenport-Hasse relation for $m=2, 3, 4$. When $ m= 2$, we have the following identity.
\begin{lemma}\emph{\label{g10}} For $A\in \widehat{\mathbb{F}_q^{\times}}$, we have
 \begin{align}
  g(A)g(\varphi A)=g(A^2)g(\varphi)\overline{A}(4).\notag
 \end{align}
\end{lemma}
For $m=3$, we have the following lemma.
\begin{lemma}\emph{\label{HD1}} Let $\chi_3$ be character of order $3$. Then for $A\in \widehat{\mathbb{F}_q^{\times}}$, we have
	\begin{align}
		g(A)g(\chi_{3}A)g(\chi_{3}^{2}A)=g(A^3)g(\chi_3)g(\chi_3^{2})\overline{A}(27).\notag
	\end{align}
\end{lemma}
For $m=4$, we have the following lemma.
\begin{lemma}\emph{\label{HD2}} Let $\chi_4$ be character of order $4$. Then for $A\in \widehat{\mathbb{F}_q^{\times}}$, we have
	\begin{align}
		g(A)g(\chi_{4}A)g(\varphi A)g(\chi_{4}^{3}A)=g(A^4)g(\chi_4)g(\varphi)g(\chi_4^{3})\overline{A}(256).\notag
	\end{align}
\end{lemma}

\begin{lemma}\emph{(\cite[(1.8)]{greene}, \cite[Thm 2.2]{mccarthy3}).}\label{g2}
For $A,B,C,D\in\widehat{\mathbb{F}_q^{\times}}$ we have
 \begin{align}
  &\frac{1}{q-1}\sum_{\chi\in\widehat{\mathbb{F}_q^{\times}}}g(A\chi)g(B\chi)g(C\overline{\chi})g(D\overline{\chi})\notag\\
  &\hspace{1cm}= \frac{g(AC)g(AD)g(BC)g(BD)}{g(ABCD)} + q(q-1)AB(-1)\delta(ABCD).\notag
 \end{align}
\end{lemma}
\begin{lemma}\cite[Thm 8.11]{FL}\label{33}
 Let $A\in\widehat{\mathbb{F}_q^\times}$ be such that $A\neq\varepsilon,\varphi$. For $x\in \mathbb{F}_q^{\times}$, we have
 \begin{align}\label{new-101}
 &{_{2}}\mathbb{F}_1\left[\begin{array}{cccc}
            A, & A\varphi \\
               & \varphi
              \end{array}\mid x \right]=\left(\frac{1+\varphi(x)}{2}\right)\left(\overline{A^2}(1+\sqrt{x})+\overline{A^2}(1-\sqrt{x})\right).
\end{align}
\end{lemma}
We remark that the formula \eqref{new-101} is well-defined and the value of the hypergeometric series will be equal to $0$ if $x$ is not a square.  
We now recall three transformation formulas of Greene.
\begin{theorem}\cite[Thm. 4.4 (i), (ii), and (iii)]{greene}\label{thm7}
For $A, B, C \in \widehat{\mathbb{F}_q^{\times}}$ and $x\in \mathbb{F}_q$, 
\begin{align}
(i)~~{_{2}}F_1\left(\begin{array}{cccc}
                A, & B\\
                 & C
              \end{array}\mid x \right)&=A(-1){_{2}}F_1\left(\begin{array}{cccc}
                A, & B\\
                 & AB\overline{C}
              \end{array}\mid 1-x \right)\notag\\
              &\hspace{.5cm}+A(-1)\displaystyle{B \choose \overline{A}C}
              \delta(1-x)-\displaystyle{B \choose C}\delta(x),\notag\\
(ii)~~{_{2}}F_1\left(\begin{array}{cccc}
                A, & B\\
                 & C
              \end{array}\mid x \right)&=C(-1)\overline{A}(1-x){_{2}}F_1\left(\begin{array}{cccc}
                A, & C\overline{B}\\
                 & C
              \end{array}\mid \frac{x}{x-1} \right)\notag\\
              &\hspace{.5cm}+A(-1)\displaystyle{B \choose \overline{A}C}\delta(1-x),\notag\\
(iii)~~{_{2}}F_1 \left(\begin{array}{ccc}
                       A, & B \\
                          & C 
                      \end{array}\mid x\right)&= \overline{B}(1-x){_{2}}F_1 \left(\begin{array}{ccc}
                                                                                   C\overline{A}, & B \\
                                                                                      & C
                                                                                  \end{array}\mid \frac{x}{x-1}\right)\notag\\
                                                                                  &\hspace{1cm}+ A(-1){B \choose \overline{A}C}\delta(1-x).\notag
\end{align}
\end{theorem}
\begin{lemma}\label{g21}
 For $A, B, C \in \widehat{\mathbb{F}_q^{\times}}$ and $x\in \mathbb{F}_q$ such  that $x\neq 0,1$, we have
\begin{align}
 {_{2}}F_1\left(\begin{array}{cccc}
                A, & B\\
                 & C
              \end{array}\mid x \right)= BC(-1)\overline{A}(x){_{2}}F_1\left(\begin{array}{cccc}
                                                                         A, & A\overline{C}\\
                                                                            & A\overline{B}
                                                            \end{array}\mid \frac{1}{x} \right).\notag
\end{align}
\begin{proof}
 Using Theorem \ref{thm7} (i) and (ii) we have 
 \begin{align}\label{g21eq1}
  {_{2}}F_1\left(\begin{array}{cccc}
                A, & B\\
                 & C
              \end{array}\mid x \right)=ABC(-1)\overline{A}(x){_{2}}F_1\left(\begin{array}{cccc}
                                                                              A, & B\\
                                                                                 & C
                                                                           \end{array}\mid \frac{x-1}{x} \right).
\end{align}
Again using Theorem \ref{thm7} (i) in \eqref{g21eq1} we complete the proof.
\end{proof}
\end{lemma}
The following lemma gives values of McCarthy's finite field hypergeometric series at $x=1$.
\begin{lemma}\label{New1}
 For $A,B, C\in\widehat{\mathbb{F}_q^{\times}}$ we have
 \begin{align}
 &{_{2}}F_1\left(\begin{array}{cc}
                A, & B\\
                 & C
              \end{array}\mid 1 \right)^{\ast} = \frac{g(A\overline{C})g(B\overline{C})}{g(\overline{C})g(AB\overline{C})}
              +\frac{q(q-1)AB(-1)}{g(A)g(B)g(\overline{C})}\delta(AB\overline{C}).\notag
 \end{align}
\end{lemma}
\begin{proof}
 The proof follows directly by using  \eqref{mcCarty's_defn} and Lemma \ref{g2}.
\end{proof}
To deduce the special values obtained by Evans and Greene from our product formulas, we need to use the fact that $A(-1)=-1$ if $A$ is a non-square character. In the following two lemmas, we prove 
this fact. We do not know if the result already exists in the literature.
\begin{lemma}\label{sq-1}
Let $A \in \widehat{\mathbb{F}_q^{\times}}$ be of order $m>1$. Then $A(-1)= -1$ if and only if $m$ is even and $\frac{q-1}{m}$ is odd.
\begin{proof}
Let $g$ be a generator of the cyclic group $\mathbb{F}_q^{\times}$. Since $m$ is the order of the character $A$, therefore $A(g)=\zeta$, a primitive $m$-th root of unity.
We have $A(-1)= A(g^{\frac{q-1}{2}})= \zeta^{\frac{q-1}{2}}$. Suppose that $A(-1)=-1$. Then $(-1)^m=A^m(-1)=1$, and hence $m$ is even. Also, $\zeta^{\frac{q-1}{2}}= A(-1)=-1= \zeta^{\frac{m}{2}}$. 
This gives $\frac{q-1}{2} \equiv \frac{m}{2}\pmod{m}$, and hence 
$\frac{q-1}{m} \equiv 1\pmod{2}$ or equivalently $\frac{q-1}{m}$ is odd. Conversely, if $m$ is even and $ \frac{q-1}{m}$ is odd then $\frac{q-1}{2} \equiv \frac{m}{2}\pmod{m}$. 
Hence, $-1=\zeta^{\frac{m}{2}}=\zeta^{\frac{q-1}{2}}$. This implies that $A(-1)= -1$.   	
\end{proof}		
\end{lemma}	
\begin{lemma}\label{sq-2}
If $A \in \widehat{\mathbb{F}_q^{\times}}$ is not a square, then $A(-1)= -1$.
\begin{proof}
Let $m$ be the order of the character $A$. Then $G=\langle A\rangle$ is a cyclic subgroup of $\widehat{\mathbb{F}_q^{\times}}$ of order $m$. Since $A$ is not a square character, so $A^2$ is not a generator of $G$. 
This implies that $\gcd(2, m)= 2$, that is $m$ is even. We next prove that $\frac{q-1}{m}$ is odd. Otherwise, $\widehat{\mathbb{F}_q^{\times}}$ will have an element of order $2m$, say $B$.
Then we must have $\langle A\rangle=\langle B^2 \rangle$. This is a contradiction to the fact that $A$ is not a square. Hence $\frac{q-1}{m}$ is odd. Using Lemma \ref{sq-1} we complete the proof of the lemma.     
	\end{proof}		
\end{lemma}
\section{Proofs of the product formulas} 
In this section we prove the product formulas satisfied by the Gaussian hypergeometric series. In \cite{TB-1} we defined a finite field analogue of the classical 
Appell series $F_4$, and proved several identities satisfied by $F_4$ over finite fields. Our work on Appell series $F_4$ plays a crucial role in the proofs of the main results of this article. 
For $A, B, C \in \widehat{\mathbb{F}_q^{\times}}$ 
and $x, y\in \mathbb{F}_q$, we define the finite field analogue of Appell series $F_4$ by    
\begin{align}\label{f4-star}
& F_{4}(A;B;C,C^{\prime};x,y)^{*}\notag\\
& = \frac{1}{(q-1)^2} \sum_{\chi, \lambda\in\widehat{\mathbb{F}_q^{\times}}}
 \frac{g(A\chi\lambda)g(B\chi\lambda)g(\overline{C\chi})g(\overline{C^{\prime}\lambda})
 g(\overline{\lambda})g(\overline{\chi})}{g(A)g(B)g(\overline{C})g(\overline{C^{\prime}})}\chi(x)\lambda(y).
\end{align} 
In \cite[Theorem 1.2]{TB-1}, we expressed finite field Appell series $F_4$ as a product of McCarthy's ${_2}F_1$-hypergeometric series under the condition that $A, B, C\neq \varepsilon$.
To deduce some interesting special values of Gaussian hypergeometric series from our product formulas, we need to allow $C=\varepsilon$. In the following theorem, we restate Theorem 1.2 of \cite{TB-1} and 
present a brief proof.
\begin{theorem}\label{New-Lemma}
Let $A, B, C \in \widehat{\mathbb{F}_q^{\times}}$ be such that $A, B\neq \varepsilon$, $B\neq C$, and $A\neq C$. 
For $x, y\in \mathbb{F}_q$ with $x, y\neq 1$,  we have
 \begin{align}
&F_{4}\left(A;B;C,AB\overline{C};\frac{-x}{(1-x)(1-y)},\frac{-y}{(1-x)(1-y)}\right)^{*}\notag\\
  &\hspace{1cm}={_{2}}F_1\left(\begin{array}{cc}
                A, & B\\
                 & C
              \end{array}\mid -\frac{x}{1-x} \right)^{\ast}{_{2}}F_1\left(\begin{array}{cc}
                A, & B\\
                 & AB\overline{C}
              \end{array}\mid -\frac{y}{1-y} \right)^{\ast}\notag\\
              &\hspace{1.5cm}-\frac{q^2AC(-1)\overline{B}C(y)A(1-x)B(1-y)}{g(A)g(B)g(\overline{C})g(\overline{AB}C)}\delta(1-xy).\notag 
\end{align}
 \end{theorem}
 \begin{proof}
  The result holds trivially if $xy=0$. Therefore, we assume that both $x$ and $y$ are nonzero. 
From \cite[Thm 1.1]{TB-1} we have
 \begin{align}
  L:=&\overline{A}(1-x)\overline{B}(1-y)F_{4}\left(A;B;C,AB\overline{C};\frac{-x}{(1-x)(1-y)},\frac{-y}{(1-x)(1-y)}\right)^{*}\notag\\
  &=\frac{1}{(q-1)^2}\sum_{\psi, \chi\in \widehat{\mathbb{F}_q^{\times}}}{_{2}}F_1\left(\begin{array}{cc}
                \overline{\chi}, & A\psi\\
                 & AB\overline{C}
              \end{array}\mid 1 \right)^{\ast}{_{2}}F_1\left(\begin{array}{cc}
                \overline{\psi}, & B\chi\\
                 & C
              \end{array}\mid 1 \right)^{\ast}\notag\\
              &\hspace{3cm}\times \frac{g(A\psi)g(B\chi)g(\overline{\chi})g(\overline{\psi})}{g(A)g(B)}\psi(-x)\chi(-y).\notag
 \end{align}
  Lemma \ref{New1} yields
\begin{align}
 L&=\frac{1}{(q-1)^2}\sum_{\psi, \chi}\left(\frac{g(\overline{AB}C\overline{\chi})g(\overline{B}C\psi)}{g(\overline{AB}C)g(C\psi\overline{B\chi})}
 +\frac{q(q-1)A\psi\chi(-1)\delta(C\psi\overline{B\chi})}{g(\overline{\chi})g(A\psi)g(\overline{AB}C)}\right) \notag \\
 &\times\left(\frac{g(\overline{C\psi})g(B\overline{C}\chi)}{g(\overline{C})g(B\chi\overline{C\psi})}
 +\frac{q(q-1)B\psi\chi(-1)\delta(B\chi\overline{C\psi})}{g(\overline{\psi})g(B\chi)g(\overline{C})}\right)\notag\\
 &\times 
 \frac{g(A\psi)g(B\chi)g(\overline{\chi})g(\overline{\psi})\psi(-x)\chi(-y)}{g(A)g(B)}\notag \\
 \label{mt2-q1}
 &=\frac{1}{(q-1)^2}\sum_{\psi, \chi}\frac{g(\overline{AB}C\overline{\chi})g(\overline{B}C\psi)g(\overline{C\psi})g(B\overline{C}\chi)g(A\psi)g(B\chi)}
 {g(A)g(B)g(\overline{C})g(\overline{AB}C)g(C\psi\overline{B\chi})g(B\chi\overline{C\psi})}  \\
 &\hspace{2cm}\times g(\overline{\chi})g(\overline{\psi})\psi(-x)\chi(-y)+\alpha_1+\alpha_2+\alpha_3,\notag
\end{align}
where
\begin{align}
 \alpha_1 &= A(-1)\frac{q}{q-1}\sum_{\psi,\chi}\frac{g(\overline{C\psi})g(B\overline{C}\chi)g(B\chi)g(\overline{\psi})}
 {g(A)g(B)g(\overline{C})g(\overline{AB}C)g(B\overline{C\psi}\chi)}\chi(y)\psi(x)\delta(\overline{B}C\overline{\chi}\psi),\notag\\
 \alpha_2&= B(-1)\frac{q}{q-1}\sum_{\psi,\chi}\frac{g(\overline{AB}C\overline{\chi})g(\overline{B}C\psi)g(A\psi)g(\overline{\chi})}
 {g(A)g(B)g(\overline{AB}C)g(\overline{B}C\overline{\chi}\psi)g(\overline{C})}\chi(y)\psi(x)\delta(B\overline{C}\overline{\psi}\chi),\notag \\
 \alpha_3&=q^2AB(-1)\sum_{\psi,\chi}\frac{\psi(-x)\chi(-y)\delta(\overline{B}C\overline{\chi}\psi)\delta(B\overline{C}\overline{\psi}\chi)}
 {g(A)g(B)g(\overline{AB}C)g(\overline{C})}.\notag
\end{align}
The above terms are nonzero only when $\overline{\chi}\psi=B\overline{C}$. So, after putting $\overline{\chi}=B\overline{C\psi}$ and using the fact that $g(\varepsilon)=-1$,
we obtain
\begin{align}\label{mt2-q2}
 \alpha_1&= -A(-1)\frac{q}{q-1}\sum_{\psi}\frac{g(\overline{C\psi})g(C\psi)g(\overline{\psi})g(\psi)}
 {g(A)g(B)g(\overline{C})g(\overline{AB}C)}\overline{B}C(y)\psi(xy),\\
 \label{mt2-q3}
 \alpha_2&= -B(-1)\frac{q}{q-1}\sum_{\psi}\frac{g(\overline{A\psi})g(A\psi)g(B\overline{C\psi})g(\overline{B}C\psi)}
 {g(A)g(B)g(\overline{C})g(\overline{AB}C)}\overline{B}C(y)\psi(xy),\\
 \alpha_3&=\frac{q^2AB(-1)\overline{B}C(-y)}{g(A)g(B)g(\overline{AB}C)g(\overline{C})}\sum_{\psi}\psi(xy).\notag
\end{align}
In case of $\alpha_3$, Lemma \ref{g5} yields
 \begin{align}\label{mt2-q4}
 \alpha_3=\frac{q^2(q-1)AC(-1)\overline{B}C(y)}{g(A)g(B)g(\overline{AB}C)g(\overline{C})}\delta(1-xy).
 \end{align}
Using Lemma \ref{g1} and  Lemma \ref{g5} we have 
\begin{align}\label{mt2-q13}
 \alpha_1 &= -\frac{q^3AC(-1)\overline{B}C(y)\delta(1-xy)}{g(A)g(B)g(\overline{C})g(\overline{AB}C)} + \frac{q^2AC(-1)\overline{B}C(y)\overline{C}(xy)}{g(A)
 g(B)g(\overline{C})g(\overline{AB}C)} \notag\\
 & \hspace{1cm}+ \frac{q^2AC(-1)\overline{B}C(y)}{g(A)g(B)g(\overline{C})g(\overline{AB}C)}- \frac{q(q-1)A(-1)\overline{B}C(y)}{g(A)g(B)g(\overline{C})g(\overline{AB}C)}\delta(C).
\end{align}
From \eqref{mt2-q1} we have
\begin{align}
 L&=\frac{1}{(q-1)^2}\sum_{\substack{\psi, \chi\\\overline{\chi}\psi\neq B\overline{C}}}\frac{g(\overline{AB}C\overline{\chi})g(\overline{B}C\psi)g(\overline{C\psi})
 g(B\overline{C}\chi)g(A\psi)g(B\chi)}{g(A)g(B)g(\overline{C})g(\overline{AB}C)g(C\psi\overline{B\chi})g(B\chi\overline{C\psi})}\notag \\
\hspace{2cm} &\times g(\overline{\chi})g(\overline{\psi})\psi(-x)\chi(-y)+\beta +\alpha_1+\alpha_2+\alpha_3,\notag
\end{align}
where
\begin{align}
\beta =\frac{BC(-1)}{(q-1)^2}\sum_{\psi}\frac{g(\overline{A\psi})g(A\psi)g(\overline{B}C\psi)g(B\overline{C\psi})g(\overline{C\psi})g(C\psi)g(\psi)
 g(\overline{\psi})}{g(A)g(B)g(\overline{C})g(\overline{AB}C)}\overline{B}C(y)\psi(xy).\notag
\end{align}
Using Lemma \ref{g1} on $g(C\psi\overline{B\chi})g(B\chi\overline{C\psi})$ we have
\begin{align}
 L&=\frac{BC(-1)}{q(q-1)^2}\sum_{\substack{\psi, \chi\\\overline{\chi}\psi\neq B\overline{C}}}\frac{g(\overline{AB}C\overline{\chi})g(\overline{B}C\psi)g(\overline{C\psi})
 g(B\overline{C}\chi)g(A\psi)g(B\chi)}{g(A)g(B)g(\overline{C})g(\overline{AB}C)} \notag\\
\hspace{2cm} &\times g(\overline{\chi})g(\overline{\psi})\psi(x)\chi(y)+\beta +\alpha_1+\alpha_2+\alpha_3\notag\\
\label{mt2-q8}
 &=\frac{BC(-1)}{q(q-1)^2}\sum_{\psi,\chi }\frac{g(\overline{AB}C\overline{\chi})g(\overline{B}C\psi)g(\overline{C\psi})g(B\overline{C}\chi)g(A\psi)g(B\chi)}
 {g(A)g(B)g(\overline{C})g(\overline{AB}C)}\\
 &\times g(\overline{\chi})g(\overline{\psi})\psi(x)\chi(y) + \frac{q-1}{q}\beta +\alpha_1+\alpha_2+\alpha_3.\notag
\end{align}
Employing Lemma \ref{g1} on $g(C\psi)g(\overline{C\psi})$ and $g(\psi)g(\overline{\psi})$ we have
\begin{align}\label{mt2-q9}
 \beta=\frac{q^2B(-1)}{(q-1)^2}\sum_{\psi}\frac{g(\overline{A\psi})g(A\psi)g(B\overline{C\psi})g(\overline{B}C\psi)}
 {g(A)g(B)g(\overline{C})g(\overline{AB}C)}\overline{B}C(y)\psi(xy) - \beta_1 -\beta_2 +\beta_3,
 \end{align}
where 
\begin{align}
 &\beta_1 = \frac{qBC(-1)}{q-1}\sum_{\psi}\frac{g(\overline{A\psi})g(A\psi)g(B\overline{C\psi})g(\overline{B}C\psi)\overline{B}C(y)\psi(xy)}
 {g(A)g(B)g(\overline{C})g(\overline{AB}C)}\psi(-1)\delta(C\psi), \notag \\ 
 &\beta_2 = \frac{qBC(-1)}{q-1}\sum_{\psi}\frac{g(\overline{A\psi})g(A\psi)g(B\overline{C\psi})g(\overline{B}C\psi)\overline{B}C(y)\psi(xy)}
 {g(A)g(B)g(\overline{C})g(\overline{AB}C)}C\psi(-1)\delta(\psi),\notag \\
 &\beta_3 = BC(-1)\sum_{\psi}\frac{g(\overline{A\psi})g(A\psi)g(B\overline{C\psi})g(\overline{B}C\psi)\overline{B}C(y)\psi(xy)}
 {g(A)g(B)g(\overline{C})g(\overline{AB}C)}\delta(C\psi)\delta(\psi).\notag 
\end{align}
Since $\beta_1$ is nonzero only when $\psi=\overline{C}$, so after putting $\psi=\overline{C}$ and then using Lemma \ref{g1} and the fact that $B, A\overline{C}\neq\varepsilon$, we obtain
\begin{align}\label{mt2-q10}
\beta_1&= \frac{q^3 AC(-1)\overline{B}C(y)\overline{C}(xy)}{(q-1)g(A)g(B)g(\overline{C})g(\overline{AB}C)}. 
\end{align}
 Similarly, 
\begin{align}\label{mt2-q11}
\beta_2&=\frac{q^3 AC(-1)\overline{B}C(y)}{(q-1)g(A)g(B)g(\overline{C})g(\overline{AB}C)}
\end{align}
and
 \begin{align}\label{New2}
  \beta_3&= \frac{q^2 A(-1)\overline{B}C(y)}{g(A)g(B)g(\overline{C})g(\overline{AB}C)}\delta(C). 
 \end{align}
 Putting  \eqref{mt2-q9} and \eqref{mt2-q3} into \eqref{mt2-q8} we obtain
\begin{align}\label{New3}
 L&=\frac{BC(-1)}{q(q-1)^2}\sum_{\psi,\chi }\frac{g(\overline{AB}C\overline{\chi})g(\overline{B}C\psi)g(\overline{C\psi})g(B\overline{C}\chi)g(A\psi)g(B\chi)}
 {g(A)g(B)g(\overline{C})g(\overline{AB}C)}\notag\\
 &\times g(\overline{\chi})g(\overline{\psi})\psi(x)\chi(y) - \frac{q-1}{q}\beta_1 - \frac{q-1}{q}\beta_2 +\frac{q-1}{q}\beta_3 +\alpha_1+\alpha_3.
\end{align}
Multiplying both numerator and denominator by $g(B\overline{C})g(\overline{B}C)$, and then using Lemma \ref{g1} and \eqref{mcCarty's_defn} we have
\begin{align}\label{new-5}
 L&={_{2}}F_1\left(\begin{array}{cc}
                A, & \overline{B}C\\
                 & C
              \end{array}\mid x \right)^{\ast}{_{2}}F_1\left(\begin{array}{cc}
                B, & B\overline{C}\\
                 & AB\overline{C}
              \end{array}\mid y \right)^{\ast} \notag\\
              &\hspace{1cm}- \frac{q-1}{q}\beta_1 - \frac{q-1}{q}\beta_2+\frac{q-1}{q}\beta_3 +\alpha_1+\alpha_3.
\end{align}  
Using Theorem \ref{thm7} (ii) in \eqref{new-5}, and then combining \eqref{mt2-q4}, \eqref{mt2-q13}, \eqref{mt2-q10},   \eqref{mt2-q11} and \eqref{New2} we find that
\begin{align*}
 L&=\overline{A}(1-x)\overline{B}(1-y){_{2}}F_1\left(\begin{array}{cc}
                A, & B\\
                 & C
              \end{array}\mid \frac{-x}{1-x} \right)^{\ast}{_{2}}F_1\left(\begin{array}{cc}
                B, & A\\
                 & AB\overline{C}
              \end{array}\mid \frac{-y}{1-y} \right)^{\ast}  \\
  &- \frac{q^2AC(-1)\overline{B}C(y)\delta(1-xy)}{g(A)g(B)g(\overline{C})g(\overline{AB}C)}. 
\end{align*}
Finally, multiplying both sides by $A(1-x)B(1-y)$ we complete the proof of the theorem. 
\end{proof}
 In the following lemma, we re-write Theorem \ref{New-Lemma} in terms of Greene's finite field hypergeometric series. 
\begin{lemma}\label{New-Lemma-1}
 Let $A, B, C \in \widehat{\mathbb{F}_q^{\times}}$ be such that $A, B\neq \varepsilon$, $B\neq C$, and  $A\neq C$. 
For $z, w\in \mathbb{F}_q$ such that $z, w\neq 1$  we have
\begin{align}
 &{_{2}}F_1\left(\begin{array}{cc}
                A, & B\\
                 & C
              \end{array}\mid z \right){_{2}}F_1\left(\begin{array}{cc}
                A, & B\\
                 & AB\overline{C}
              \end{array}\mid w \right)\notag\\
              &=\frac{A(-1)g(B)g(\overline{C})g(\overline{AB}C)}{qg(\overline{B})g(B\overline{C})g(\overline{A}C)}F_{4}\left(A;B;C,AB\overline{C};z(1-w),w(1-z)\right)^{*} \notag\\
&\hspace{1cm}+\frac{qB(-1)\overline{A}(1-z)\overline{B}C(w)\overline{C}(1-w)}{g(A)g(\overline{B})g(B\overline{C})g(\overline{A}C)}\delta\left(\frac{1-w-z}{(1-z)(1-w)}\right).\notag
\end{align}
\begin{proof}
Using \eqref{prop-300} in Theorem \ref{New-Lemma} we have
\begin{align}\label{New Lemma-1eq1}
&{_{2}}F_1\left(\begin{array}{cc}
                A, & B\\
                 & C
              \end{array}\mid -\frac{x}{1-x} \right){_{2}}F_1\left(\begin{array}{cc}
                A, & B\\
                 & AB\overline{C}
              \end{array}\mid -\frac{y}{1-y} \right)\notag\\
&= {B\choose C}{B\choose AB\overline{C}}\times\left[F_{4}\left(A;B;C,AB\overline{C};\frac{-x}{(1-x)(1-y)},\frac{-y}{(1-x)(1-y)}\right)^{*}\right.\notag\\ &\hspace{3.5cm}\left.+\frac{q^2AC(-1)\overline{B}C(y)A(1-x)B(1-y)}{g(A)g(B)g(\overline{C})g(\overline{AB}C)}\delta(1-xy)
 \right].           
\end{align}
Applying Lemma \ref{g8} and Lemma \ref{g1}, and then putting $z=\frac{-x}{1-x}$ and $w=\frac{-y}{1-y}$ in \eqref{New Lemma-1eq1}, we complete the proof.
\end{proof}
\end{lemma}
We now present a proof of our main product formula Theorem \ref{MT41}.
\begin{proof}[Proof of Theorem \ref{MT41}] The result is trivially true if $x=0$. So, we assume that $x\neq 0$.
Let
\begin{align}
  L:&= {_{2}}F_1\left(\begin{array}{cccc}
            A^2, & B^2 \\
                 & C
              \end{array}\mid x \right){_{2}}F_1\left(\begin{array}{cccc}
                                                                       A^2, & B^2 \\
                                                                            & A^2B^2\overline{C}
                                                     \end{array}\mid x \right).\notag
\end{align}
Using Lemma \ref{New-Lemma-1}, we have 
\begin{align}\label{mt41eq1}
L&=\frac{g(B^2)g(\overline{C})g(\overline{A^2B^2}C)}{qg(\overline{B^2})g(B^2\overline{C})g(\overline{A^2}C)}F_{4}\left(A^2;B^2;C,A^2B^2\overline{C};x(1-x), x(1-x)\right)^{*}+ I_1,                                                     
\end{align}
where
\begin{align}\label{mt41eq3}
I_1&=\frac{q\overline{C}\overline{A^2}(1-x)C\overline{B^2}(x)}{g(A^2)g(\overline{B^2})g(B^2\overline{C})g(\overline{A^2}C)}\delta\left(\frac{1-2x}{(1-x)^2}\right).
\end{align}
Now employing \eqref{f4-star} into \eqref{mt41eq1} we have
\begin{align}
L&=\frac{1}{q(q-1)^2} \sum_{\chi, \lambda\in\widehat{\mathbb{F}_q^{\times}}}
  \frac{g(A^2\chi\lambda)g(B^2\chi\lambda)g(\overline{C\chi})g(\overline{A^2B^2}C\overline{\lambda})g(\overline{\lambda})g(\overline{\chi})}{g(A^2)g(\overline{B^2})
  g(B^2\overline{C})g(\overline{A^2}C)}\chi\lambda(x-x^2)+ I_1. \notag                                                    
\end{align}
The change of variables $\chi\mapsto\chi\overline{\lambda}$ yield 
 \begin{align}
  L&=\frac{1}{q(q-1)^2}\sum_{\chi, \lambda\in\widehat{\mathbb{F}_q^{\times}}}
  \frac{g(A^2\chi)g(B^2\chi)g(\overline{C\chi}\lambda)g(\overline{A^2B^2}C\overline{\lambda})g(\overline{\lambda})g(\overline{\chi}\lambda)}{g(A^2)g(\overline{B^2})g(B^2\overline{C})
  g(\overline{A^2}C)}\chi(x-x^2) + I_1.\notag
 \end{align}
Using Lemma \ref{g2}, we have
 \begin{align}\label{mt41eq5}
   L&=\frac{1}{q(q-1)}\sum_{\chi\in\widehat{\mathbb{F}_q^{\times}}}
  \frac{g(A^2\chi)g(B^2\chi)g(\overline{A^2B^2\chi})g(\overline{A^2B^2}C\overline{\chi})g(\overline{C\chi})g(\overline{\chi})}{g(A^2)g(\overline{B^2})g(B^2\overline{C})g(\overline{A^2}C)
  g(\overline{A^2B^2\chi^2})}\chi(x-x^2) + I_1 + I_2,
\end{align}
where
  \begin{align}\label{mt41eq6}
   I_2&= C(-1)\sum_{\chi\in\widehat{\mathbb{F}_q^{\times}}}
 \frac{ g(A^2\chi)g(B^2\chi)\chi(x-x^2)}{g(A^2)g(\overline{B^2})g(B^2\overline{C})g(\overline{A^2}C)}\delta(\overline{A^2B^2\chi^2})\notag\\
  &=C(-1)\left[\frac{g(A\overline{B})g(\overline{A}B)\overline{AB}(x-x^2)+ g(\overline{A}B\varphi)g(A\overline{B}\varphi)\overline{AB}\varphi(x-x^2)}{g(A^2)g(\overline{B^2})g(B^2\overline{C})g(\overline{A^2}C)}\right]. 
  \end{align}
 The last equality is obtained by putting $\chi= \overline{AB}, \overline{AB}\varphi$. Now using Lemma \ref{g3} in \eqref{mt41eq5}, we have 
  \begin{align}\label{mt41eq7}
    L&=\frac{1}{q^2(q-1)} \sum_{\chi\in\widehat{\mathbb{F}_q^{\times}}}
   \frac{g(A^2\chi)g(B^2\chi)g(\overline{C\chi})g(\overline{A^2B^2\chi})g(\overline{A^2B^2}C\overline{\chi})}{g(A^2)g(\overline{B^2})g(B^2\overline{C})g(\overline{A^2}C)}\notag\\
  &\hspace{3cm}\times g(A^2B^2\chi^2)g(\overline{\chi})\chi(x-x^2) + I_1 + I_2 - I_3,
   \end{align}
 where
 \begin{align}\label{mt41eq8}
    I_3&= \frac{1}{q^2} \sum_{\chi\in\widehat{\mathbb{F}_q^{\times}}}\frac{g(A^2\chi)g(B^2\chi)g(\overline{C\chi})g(\overline{A^2B^2\chi})
     g(\overline{A^2B^2}C\overline{\chi})g(\overline{\chi})\chi(x-x^2)}{g(A^2)g(\overline{B^2})g(B^2\overline{C})g(\overline{A^2}C)}\delta(A^2B^2\chi^2)\notag\\
  &= I_2 + \frac{(q-1)g(A\overline{B})g(\overline{A}B)\overline{AB}(x-x^2)}{q^2g(A^2)g(\overline{B^2})g(B^2\overline{C})g(\overline{A^2}C)}\left[\{(q-1)\delta(AB)
  - qAB(-1)\}\delta(AB\overline{C})\right.\notag\\
  &\left.-qABC(-1)\delta(AB)\right] 
  + \frac{(q-1)g(\overline{A}B\varphi)g(A\overline{B}\varphi)\overline{AB}\varphi(x-x^2)}{q^2g(A^2)g(\overline{B^2})g(B^2\overline{C})g(\overline{A^2}C)} \notag\\
  &\hspace{.5cm}\times \left[(q-1)\delta(AB\overline{C}\varphi)\delta(AB\varphi)-qAB\varphi(-1)\delta(AB\overline{C}\varphi) -q ABC\varphi(-1)\delta(AB\varphi) \right].
   \end{align}
   The last equality is obtained by putting $\chi= \overline{AB}, \overline{AB}\varphi$ and using Lemma \ref{g1} on $g(AB)g(\overline{AB}), g(AB\overline{C})g(\overline{AB}C)$, 
   $g(AB\varphi)g(\overline{AB}\varphi)$ and $g(AB\overline{C}\varphi)g(\overline{AB}C\varphi).$ Now using Lemma \ref{g10} on $g(A^2B^2\chi^2)$, \eqref{mt41eq7} reduces to
   \begin{align}
    L&=\frac{AB(4)}{q^2(q-1)} \sum_{\chi\in\widehat{\mathbb{F}_q^{\times}}}
   \frac{g(A^2\chi)g(B^2\chi)g(\overline{C\chi})g(\overline{A^2B^2\chi})g(\overline{A^2B^2}C\overline{\chi})g(AB\chi)}{g(A^2)g(\overline{B^2})g(B^2\overline{C})g(\overline{A^2}C)g(\varphi)}\notag\\
   &\hspace{3cm}\times g(AB\varphi\chi)g(\overline{\chi})\chi(4x-4x^2) + I_1 + I_2 - I_3.\notag
   \end{align}
Multiplying both numerator and denominator by $q^4g(\overline{A^2})g(AB\overline{C})g(\overline{AB}C\varphi)$ and then rearranging the terms  we have
\begin{align}
    L&=\frac{q^2AB(4)g(\overline{A^2})g(AB\overline{C})g(\overline{AB}C\varphi)}{(q-1)g(\overline{B^2})g(B^2\overline{C})g(\overline{A^2}C)g(\varphi)} \notag\\
    &\times \sum_{\chi\in\widehat{\mathbb{F}_q^{\times}}}\left(
   \frac{g(A^2\chi)g(\overline{\chi})\chi(-1)}{qg(A^2)}\times \frac{g(B^2\chi)g(\overline{A^2B^2\chi})\chi(-1)}{qg(\overline{A^2})}\times \frac{g(AB\chi)g(\overline{C\chi})C\chi(-1)}{qg(AB\overline{C})}\right.\notag\\
   &\left. \times \frac{g(AB\varphi\chi)g(\overline{A^2B^2}C\overline{\chi})C\chi(-1)}{qg(\overline{AB}C\varphi)}\chi(4x-4x^2)\right) + I_1 + I_2 - I_3.\notag
   \end{align}
Lemma \ref{g8} yields
\begin{align}
   L&=\frac{q^2AB(4)g(\overline{A^2})g(AB\overline{C})g(\overline{AB}C\varphi)}{(q-1)g(\overline{B^2})g(B^2\overline{C})g(\overline{A^2}C)g(\varphi)} 
   \sum_{\chi\in\widehat{\mathbb{F}_q^{\times}}}{A^2\chi\choose \chi}\left({AB\chi\choose C\chi}-\frac{q-1}{q}\delta(AB\overline{C}) \right)\notag\\
   &\times {B^2\chi\choose A^2B^2\chi}\left({AB\varphi\chi\choose A^2B^2\overline{C}\chi}-\frac{q-1}{q}\delta(\overline{AB}C\varphi) \right)\chi(4x-4x^2) + I_1 + I_2 - I_3.\notag
\end{align}
Using \eqref{Greene-def-4} we have 
\begin{align}\label{mt41eq9}
   L&=\frac{qAB(4)g(\overline{A^2})g(AB\overline{C})g(\overline{AB}C\varphi)}{g(\overline{B^2})g(B^2\overline{C})g(\overline{A^2}C)g(\varphi)} {_{4}}F_3\left(\begin{array}{cccccc}
             \hspace{-.1cm}A^2, & \hspace{-.2cm}B^2, & \hspace{-.2cm} AB, & \hspace{-.2cm} AB\varphi \\
                &\hspace{-.2cm} A^2B^2, & \hspace{-.2cm} C, & \hspace{-.2cm} A^2B^2\overline{C} 
            \end{array}\mid 4x(1-x) \right)\notag\\
&\hspace{1.5cm}-\frac{(q-1)AB(4)g(\overline{A^2})g(AB\overline{C})g(\overline{AB}C\varphi)}{g(\overline{B^2})g(B^2\overline{C})g(\overline{A^2}C)g(\varphi)}\notag\\
&\hspace{1.5cm} \times \left[{_{3}}F_2\left(\begin{array}{cccccc}
           A^2, & B^2, & AB\varphi \\
                & A^2B^2, & A^2B^2\overline{C} 
            \end{array}\mid 4x(1-x) \right)\delta(AB\overline{C})\right.\notag\\
&\left.\hspace{1.5cm} + ~{_{3}}F_2\left(\begin{array}{cccccc}
            A^2, & B^2, & AB \\
                & A^2B^2, & C 
            \end{array}\mid 4x(1-x) \right)\delta(\overline{AB}C\varphi) \right] + I_1 + I_2 - I_3.
\end{align}
Finally employing \eqref{mt41eq3}, \eqref{mt41eq6}, and \eqref{mt41eq8} into \eqref{mt41eq9}, we complete the proof of the theorem. 
 \end{proof} 
 We now state a special case of Theorem \ref{MT41} which will be used to prove Theorem \ref{MT41C2} and to derive certain special values.
\begin{corollary}\label{MT41C1}
Let $A, B  \in \widehat{\mathbb{F}_q^{\times}}$ be such that $A^2, B^2, A\overline{B}\varphi\neq\varepsilon$. For $x\neq 1$ we have 
\begin{align}
  &{_{2}}F_1\left(\begin{array}{cccc}
            A^2, & B^2 \\
                 & AB\varphi
              \end{array}\mid x \right)^2
= \frac{qAB(4)g(\overline{A^2})}{g(\overline{B^2})g(\overline{A}B\varphi)^2}
{_{3}}F_2\left(\begin{array}{cccccc}
                       \hspace{-.1cm} A^2, & \hspace{-.1cm} B^2, & \hspace{-.1cm} AB \\
                            & \hspace{-.1cm} A^2B^2, & \hspace{-.1cm} AB\varphi 
                        \end{array}\mid 4x(1-x) \right)\notag\\
&+\frac{g(AB\varphi)g(\overline{AB}\varphi)g(A\overline{B}\varphi)}{q^2g(\overline{B^2})g(A^2)g(\overline{A}B\varphi)}\overline{AB}\varphi(x-x^2)
+\frac{q\overline{A^3B}\varphi(1-x)A\overline{B}\varphi(x)}{g(A^2)g(\overline{B^2})g(\overline{A}B\varphi)^2}\delta\left(\frac{1-2x}{(x-1)^2}\right)\notag\\
&+\frac{(q-1)\varphi(-1)g(A\overline{B})g(\overline{A}B)\overline{AB}(x-x^2)}{qg(A^2)g(\overline{B^2})g(\overline{A}B\varphi)^2}\delta(AB)\notag\\
&+\frac{(q-1)g(A\overline{B}\varphi)\overline{AB}\varphi(x-x^2)}{q^2g(A^2)g(\overline{B^2})g(\overline{A}B\varphi)}\left[\delta(AB\varphi)+qAB\varphi(-1)\right].\notag
\end{align}
\end{corollary}
  \begin{proof}
  The result is trivially true if $x=0$. So, let $x\neq 0$.
   Putting $C= AB\varphi$ in Theorem \ref{MT41} and using the fact that $g(\varepsilon)= -1$ we have
   \begin{align}\label{mt41c1eq1}
 &{_{2}}F_1\left(\begin{array}{cccc}
            A^2, & B^2 \\
                 & AB\varphi
              \end{array}\mid x \right)^2\notag\\ 
&= -\frac{qAB(4)g(\overline{A^2})}{g(\overline{B^2})g(\overline{A}B\varphi)^2} {_{4}}F_3\left(\begin{array}{cccccc}
    A^2, & B^2, & AB, & AB\varphi \\
        & A^2B^2, & AB\varphi, & AB\varphi
    \end{array}\mid 4x(1-x) \right)\notag\\
&+\frac{(q-1)AB(4)g(\overline{A^2})}{g(\overline{B^2})g(\overline{A}B\varphi)^2}{_{3}}F_2\left(\begin{array}{cccccc}
    A^2, & B^2, & AB \\
        & A^2B^2, & AB\varphi
    \end{array}\mid 4x(1-x) \right)\notag\\    
&+\frac{q\overline{A^3B}\varphi(1-x)A\overline{B}\varphi(x)}{g(A^2)g(\overline{B^2})g(\overline{A}B\varphi)^2}\delta\left(\frac{1-2x}{(1-x)^2}\right)
+\frac{(q-1)g(A\overline{B})g(\overline{A}B)\overline{AB}(x-x^2)}{q\varphi(-1)g(A^2)g(\overline{B^2})g(\overline{A}B\varphi)^2}\delta(AB)\notag\\ 
  &+ \frac{(q-1)g(A\overline{B}\varphi)\overline{AB}\varphi(x-x^2)}{q^2g(A^2)g(\overline{B^2})g(\overline{A}B\varphi)} \left[\delta(AB\varphi)+qAB\varphi(-1) \right].
\end{align}
  Using \eqref{Greene-def-4} and \eqref{b5} we have
  \begin{align}\label{mt41c1eq2}
&-\frac{qAB(4)g(\overline{A^2})}{g(\overline{B^2})g(\overline{A}B\varphi)^2}{_{4}}F_3\left(\begin{array}{cccccc}
                         A^2, & B^2, & AB, & AB\varphi \\
                           & A^2B^2, &  AB\varphi, & AB\varphi 
                        \end{array}\mid 4x(1-x) \right)\notag\\
 &= \frac{AB(4)g(\overline{A^2})}{g(\overline{B^2})g(\overline{A}B\varphi)^2}{_{3}}F_2\left(\begin{array}{cccccc}
                         A^2, & B^2, & AB \\
                           & A^2B^2, &  AB\varphi
                        \end{array}\mid 4x(1-x) \right) - I_1,                      
 \end{align}
 where
  \begin{align}\label{mt41c1eq3}
   I_1 &= \frac{qAB(4)g(\overline{A^2})}{g(\overline{B^2})g(\overline{A}B\varphi)^2}
   \sum_{\chi\in\widehat{\mathbb{F}_q^{\times}}}{A^2\chi\choose\chi}{B^2\chi\choose A^2B^2\chi}{AB\chi\choose AB\varphi\chi}\chi(4x-4x^2)\delta(AB\varphi\chi)\notag\\
   &= -\frac{g(AB\varphi)g(\overline{AB}\varphi)g(A\overline{B}\varphi)}{q^2g(\overline{B^2})g(A^2)g(\overline{A}B\varphi)}\overline{AB}\varphi(x-x^2).
  \end{align}
The last equality is obtained by putting $\chi= \overline{AB}\varphi$, and then using Lemma \ref{g8} and $g(\varepsilon)=-1$. 
Finally, combining \eqref{mt41c1eq1}, \eqref{mt41c1eq2} and \eqref{mt41c1eq3}, we complete the proof.
\end{proof}
 \begin{proof}[Proof of Theorem \ref{MT41C2}] From \cite[(4.33)]{greene-thesis}, we have 
\begin{align}\label{mt41c2eq1}
	{_{2}}F_1\left(\begin{array}{cccc}
		A^2, & B^2 \\
		& AB\varphi
	\end{array}\mid x \right)= \frac{B(-1)g(B^2)g(A\overline{B}\varphi)}{g(B)g(A\varphi)} {_{2}}F_1\left(\begin{array}{cccc}
		A, & B \\
		& AB\varphi
	\end{array}\mid 4x(1-x) \right).
\end{align}
Using the given conditions $x\neq 1, \frac{1}{2}$ and $AB, AB\varphi\neq \varepsilon$, Corollary \ref{MT41C1} yields
\begin{align}\label{revision-1}
&{_{2}}F_1\left(\begin{array}{cccc}
A^2, & \hspace{-.1cm}B^2 \\
&\hspace{-.1cm} AB\varphi
\end{array}\mid x \right)^2
= \frac{qAB(4)g(\overline{A^2})}{g(\overline{B^2})g(\overline{A}B\varphi)^2}
{_{3}}F_2\left(\begin{array}{cccccc}
A^2, & \hspace{-.1cm} B^2, &\hspace{-.1cm} AB \\
&\hspace{-.1cm} A^2B^2, &\hspace{-.1cm} AB\varphi 
\end{array}\mid 4x(1-x) \right)\notag\\
&+\frac{g(AB\varphi)g(\overline{AB}\varphi)g(A\overline{B}\varphi)}{q^2g(\overline{B^2})g(A^2)g(\overline{A}B\varphi)}\overline{AB}\varphi(x-x^2)
+\frac{(q-1)g(A\overline{B}\varphi)\overline{AB}\varphi(x-x^2)}{qg(A^2)g(\overline{B^2})g(\overline{A}B\varphi)}AB\varphi(-1).
\end{align}
Combining \eqref{mt41c2eq1} and \eqref{revision-1}, and then employing Lemma \ref{g1} we complete the proof.
 \end{proof}
 \begin{proof}[Proof of Theorem \ref{MT42}] The result is trivially true if $z=0$. Let $x\neq 0$.
 Putting $C= D^2$ and $B= D\overline{E}$ in Theorem \ref{MT41}, we have
 \begin{align}\label{mt42eq1}
  &{_{2}}F_1\left(\begin{array}{cccc}
            A^2, & D^2\overline{E^2} \\
                  & D^2
               \end{array}\mid x \right){_{2}}F_1\left(\begin{array}{cccc}
                                                                 A^2, & D^2\overline{E^2} \\
                                                                      & A^2\overline{E^2}\end{array}\mid x \right)\notag \\
&= {_{4}}F_3\left(\begin{array}{cccccc}
                                                 A^2, & D^2\overline{E^2}, & AD\overline{E}, & AD\overline{E}\varphi \\
                                                      & A^2D^2\overline{E^2}, & D^2, & A^2\overline{E^2} 
                                               \end{array}\mid 4x(1-x) \right)\notag\\                                               
 &\times \frac{qAD\overline{E}(4)g(\overline{A^2})g(A\overline{ED})g(\overline{A}ED\varphi)}{g(\overline{D^2}E^2)g(\overline{E^2})g(\overline{A^2}D^2)g(\varphi)} +\frac{q\overline{A^2D^2}(1-x)E^2(x)}{g(A^2)g(\overline{D^2}E^2)g(\overline{E^2})g(\overline{A^2}D^2)}\delta\left(\frac{1-2x}{(x-1)^2}\right).
 \end{align}
  Using \eqref{prop-300} we find that 
 \begin{align}\label{mt42eq2}
  &{_{2}}F_1\left(\begin{array}{cccc}
            A^2, & D^2\overline{E^2} \\
                 & D^2
               \end{array}\mid x \right){_{2}}F_1\left(\begin{array}{cccc}
                                                                 A^2, & D^2\overline{E^2} \\
                                                                      & A^2\overline{E^2}\end{array}\mid x \right)\notag \\
 &={D^2\overline{E^2}\choose A^2\overline{E^2}}{A^2\choose A^2\overline{E^2}}^{-1}{_{2}}F_1\left(\begin{array}{cccc}
             \hspace{-.12cm}A^2, &\hspace{-.12cm} D^2\overline{E^2} \\
                  &\hspace{-.12cm} D^2
               \end{array}\mid x \right){_{2}}F_1\left(\begin{array}{cccc}
                                                                \hspace{-.12cm} D^2\overline{E^2}, &\hspace{-.12cm} A^2 \\
                                                                      &\hspace{-.12cm} A^2\overline{E^2}\end{array}\mid x \right)\notag\\
&=\frac{g(D^2\overline{E^2})g(E^2)}{g(A^2)g(\overline{A^2}D^2)}\overline{A^2D^2}E^2(1-x)~
 {_{2}}F_1\left(\begin{array}{cccc}
             A^2, & E^2 \\
                  & D^2
               \end{array}\mid \frac{x}{x-1} \right)\notag\\
           &\times {_{2}}F_1\left(\begin{array}{cccc}
                            D^2\overline{E^2}, & \overline{E^2} \\
                            & A^2\overline{E^2}\end{array}\mid \frac{x}{x-1} \right).                                                                      
 \end{align}
 The last equality is obtained by using Lemma \ref{g8} and  Theorem \ref{thm7} (ii).
Now using \eqref{mt42eq2} in \eqref{mt42eq1} and Lemma \ref{g1} we have
\begin{align}
&{_{2}}F_1\left(\begin{array}{cccc}
             A^2, & E^2 \\
                  & D^2
               \end{array}\mid \frac{x}{x-1} \right){_{2}}F_1\left(\begin{array}{cccc}
                            D^2\overline{E^2}, & \overline{E^2} \\
                            & A^2\overline{E^2}\end{array}\mid \frac{x}{x-1} \right)\notag\\
&= \frac{AD\overline{E}(4)A^2D^2\overline{E^2}(1-x)g(A\overline{ED})g(\overline{A}ED\varphi)}{g(\varphi)}\notag\\
&\times {_{4}}F_3\left(\begin{array}{cccccc}
                A^2, & D^2\overline{E^2}, & AD\overline{E}, & AD\overline{E}\varphi \\
                    & A^2D^2\overline{E^2}, & D^2, & A^2\overline{E^2} 
                \end{array}\mid 4x(1-x) \right)\notag\\                                               
&+\frac{\overline{E^2}(1-x)E^2(x)}{q}\delta\left(\frac{1-2x}{(x-1)^2}\right).\notag                            
\end{align}
Finally, putting $z= \frac{x}{x-1}$, we complete the proof of the theorem.
\end{proof}
 \begin{proof}[Proof of Theorem \ref{MT43}]
  The result is trivially true if $x=0$. So, let $x\neq 0$. Let 
  \begin{align}
   L:= {_{2}}F_1\left(\begin{array}{cccc}
            A, & B \\
                  & C^2
               \end{array}\mid x \right){_{2}}F_1\left(\begin{array}{cccc}
                                                                        A, & C^2\overline{B} \\
                                                                             & C^2
\end{array}\mid x \right).\notag
  \end{align}
 Using Theorem \ref{thm7} (i) and (ii) we have
 \begin{align}\label{mt43eq1}
   L&= A(-1)\overline{A}(1-x){_{2}}F_1\left(\begin{array}{cccc}
                                  A, & B \\
                                     & C^2
                                 \end{array}\mid x \right)
                             {_{2}}F_1\left(\begin{array}{cccc}
                                                      A, & B \\
                                                         & AB\overline{C^2}
                                 \end{array}\mid \frac{1}{1-x} \right).
  \end{align}
Employing Lemma \ref{New-Lemma-1} into \eqref{mt43eq1} yields
   \begin{align}\label{mt43eq2}
   L&= \frac{\overline{A}(1-x)g(B)g(\overline{C^2})g(\overline{AB}C^2)}{qg(\overline{B})g(B\overline{C^2})g(\overline{A}C^2)}F_{4}\left(A;B;C^2,AB\overline{C^2};\frac{x^2}{x-1}, 1\right)^{*}+ I_1, 
  \end{align}
  where
  \begin{align}\label{mt43eq14}
I_1&=\frac{qAB(-1)\overline{A^2}B(1-x)\overline{C^2}(x)}{g(A)g(\overline{B})g(\overline{A}C^2)g(B\overline{C^2})}\delta\left(\frac{x-2}{x-1}\right).
\end{align}
Using \eqref{f4-star} in \eqref{mt43eq2} we obtain
   \begin{align}
   L&=\frac{\overline{A}(1-x)}{q(q-1)^2} \sum_{\chi, \lambda\in\widehat{\mathbb{F}_q^{\times}}}
   \frac{g(A\chi\lambda)g(B\chi\lambda)g(\overline{C^2\chi})g(\overline{AB}C^2\overline{\lambda})
 g(\overline{\lambda})g(\overline{\chi})}{g(A)g(\overline{B})g(B\overline{C^2})g(\overline{A}C^2)}\chi\left(\frac{x^2}{x-1}\right)+ I_1.\notag
  \end{align}
  Using Lemma \ref{g2} yields
  \begin{align}\label{mt43eq3}
   L&=\frac{\overline{A}(1-x)}{q(q-1)} \sum_{\chi\in\widehat{\mathbb{F}_q^{\times}}}
   \frac{g(A\chi)g(B\chi)g(\overline{C^2\chi})g(\overline{A}C^2\chi)g(\overline{B}C^2\chi)
  g(\overline{\chi})}{g(A)g(\overline{B})g(B\overline{C^2})g(\overline{A}C^2)g(C^2\chi^2)}\chi\left(\frac{x^2}{x-1}\right)+ I_1 + I_2,
  \end{align}
  where
   \begin{align}\label{mt43eq4}
      I_2&= AB(-1)\overline{A}(1-x) \sum_{\chi\in\widehat{\mathbb{F}_q^{\times}}}
      \frac{g(\overline{C^2\chi})g(\overline{\chi})}{g(A)g(\overline{B})g(B\overline{C^2})g(\overline{A}C^2)}\chi\left(\frac{x^2}{x-1}\right)\delta(C^2\chi^2)\notag\\
      &=\frac{qAB(-1)\overline{A}(1-x)\overline{C^2}(x)C(1-x)}{g(A)g(\overline{B})g(B\overline{C^2})g(\overline{A}C^2)}\left[\varphi(1-x) + 1\right].
    \end{align}
 The last equality is obtained by putting $\chi= \overline{C}, \overline{C}\varphi$, and then using Lemma \ref{g1} and the fact that $C^2\neq\varepsilon$.  
 Now using Lemma \ref{g3} in \eqref{mt43eq3} we have
  \begin{align}\label{mt43eq5}
    L&=\frac{\overline{A}(1-x)}{q^2(q-1)} \sum_{\chi\in\widehat{\mathbb{F}_q^{\times}}}
    \frac{g(A\chi)g(B\chi)g(\overline{C^2\chi})g(\overline{A}C^2\chi)g(\overline{B}C^2\chi)g(\overline{C^2\chi^2})
  g(\overline{\chi})}{g(A)g(\overline{B})g(B\overline{C^2})g(\overline{A}C^2)}\chi\left(\frac{x^2}{x-1}\right)\notag\\
  &\hspace{2cm}+ I_1 + I_2 - I_3,
    \end{align}
   where
   \begin{align}
    I_3&= \frac{\overline{A}(1-x)}{q^2} \sum_{\chi\in\widehat{\mathbb{F}_q^{\times}}}
   \frac{g(A\chi)g(B\chi)g(\overline{C^2\chi})g(\overline{A}C^2\chi)g(\overline{B}C^2\chi)
   g(\overline{\chi})}{g(A)g(\overline{B})g(B\overline{C^2})g(\overline{A}C^2)}\chi\left(\frac{x^2}{x-1}\right)\delta(\overline{C^2\chi^2})\notag\\
   &= \frac{\overline{A}(1-x)\overline{C}(x^2)C(x-1)}{q^2 g(A)g(\overline{B})g(B\overline{C^2})g(\overline{A}C^2)}
   \left[g(A\overline{C})g(B\overline{C})g(\overline{C})g(\overline{A}C)g(\overline{B}C)g(C)\right.\notag\\
   &\left. + g(A\overline{C}\varphi)g(B\overline{C}\varphi)g(\overline{C}\varphi)g(\overline{A}C\varphi)g(\overline{B}C\varphi)g(C\varphi)\varphi(x-1)\right].\notag
   \end{align}
 The last equality is obtained by putting $\chi= \overline{C}, \overline{C}\varphi$. Using Lemma \ref{g1} on 
 $g(A\overline{C})g(\overline{A}C)$, $g(B\overline{C})g(\overline{B}C)$, $g(A\overline{C}\varphi)g(\overline{A}C\varphi)$, $g(B\overline{C}\varphi)g(\overline{B}C\varphi)$, $g(C)g(\overline{C})$ and 
 $g(C\varphi)g(\overline{C}\varphi)$ with the fact that  $C^2\neq\varepsilon$, we have
 \begin{align}\label{mt43eq6}
 	I_3&= I_2 + \frac{(q-1)\overline{A}(1-x)\overline{C}(x^2)C(1-x)}{q g(A)g(\overline{B})g(B\overline{C^2})g(\overline{A}C^2)}\left[(q-1)\delta(A\overline{C}) 
 	\delta(B\overline{C})-qBC(-1)\delta(A\overline{C})\right.\notag\\
 	&\hspace{1cm}\left.-qBC(-1)\varphi(x-1)\delta(A\overline{C}\varphi) +(q-1)\varphi(1-x)\delta(A\overline{C}\varphi)\delta(B\overline{C}\varphi)\right.\notag\\
 	&\hspace{1cm}\left.
 	-qAC(-1)\delta(B\overline{C})
 	-qAC(-1)\varphi(x-1)\delta(B\overline{C}\varphi)\right].
  \end{align}
   Now using Lemma \ref{g10}  in \eqref{mt43eq5} we have
 \begin{align}
    L&=\frac{\overline{A}(1-x)\overline{C}(4)}{q^2(q-1)} \sum_{\chi\in\widehat{\mathbb{F}_q^{\times}}}
    \frac{g(A\chi)g(B\chi)g(\overline{C^2\chi})g(\overline{A}C^2\chi)g(\overline{B}C^2\chi)g(\overline{C\chi})g(\varphi\overline{C\chi})
  g(\overline{\chi})}{g(A)g(\overline{B})g(B\overline{C^2})g(\overline{A}C^2)g(\varphi)}\notag\\
&\hspace{3.5cm}\times \chi\left(\frac{x^2}{4(x-1)}\right)+I_1 + I_2 - I_3.\notag
    \end{align}
 Multiplying both numerator and denominator by $q^4g(\overline{A}C)g(\overline{B}C\varphi)\varphi(-1)$ and then rearranging the terms  we have
 \begin{align}\label{mt43eq7}
    L&=\frac{q^2\varphi(-1)\overline{C}(4)\overline{A}(1-x)g(\overline{A}C)g(\overline{B}C\varphi)}{(q-1)g(\varphi)g(\overline{A}C^2)g(\overline{B})} \notag\\
    &\times \sum_{\chi\in\widehat{\mathbb{F}_q^{\times}}}\left(
    \frac{g(A\chi)g(\overline{\chi})\chi(-1)}{qg(A)}\right)\left(
    \frac{g(B\chi)g(\overline{C^2\chi})\chi(-1)}{qg(B\overline{C^2})}\right)\left(
    \frac{g(\overline{A}C^2\chi)g(\overline{C\chi}) C\chi(-1)}{qg(\overline{A}C)}\right)\notag\\
    &\hspace{1cm}\times \left(
    \frac{g(\overline{B}C^2\chi)g(\varphi\overline{C\chi})\varphi C\chi(-1)}{qg(\overline{B}C\varphi)}\right)\chi\left(\frac{x^2}{4x-4}\right) + I_1 + I_2 - I_3.
   \end{align}
Using  Lemma \ref{g8} and the fact that $A, B\overline{C^2}\neq\varepsilon$ in \eqref{mt43eq7} we have
 \begin{align}
   L&=\frac{q^2\varphi(-1)\overline{C}(4)\overline{A}(1-x)g(\overline{A}C)g(\overline{B}C\varphi)}{(q-1)g(\varphi)g(\overline{A}C^2)g(\overline{B})}\sum_{\chi\in\widehat{\mathbb{F}_q^{\times}}}{A\chi\choose \chi}
   \left[{\overline{A}C^2\chi\choose C\chi} -\frac{q-1}{q}\delta(\overline{A}C)\right]\notag\\
   &\times {B\chi\choose C^2\chi}\left[{\overline{B}C^2\chi\choose C\varphi\chi}-\frac{q-1}{q}\delta(\overline{B}C\varphi)\right]\chi\left(\frac{x^2}{4x-4}\right) + I_1 + I_2 - I_3.\notag
 \end{align}
Employing \eqref{Greene-def-4} yields 
\begin{align}\label{mt43eq19}
 L&=\frac{q\varphi(-1)\overline{C}(4)\overline{A}(1-x)g(\overline{A}C)g(\overline{B}C\varphi)}{g(\varphi)g(\overline{A}C^2)g(\overline{B})}{_{4}}F_3\left(\begin{array}{cccccc}
                                                \hspace{-.1cm} A, &\hspace{-.1cm} B, &\hspace{-.1cm}  \overline{A}C^2, &\hspace{-.1cm} \overline{B}C^2 \\
                                                     &\hspace{-.1cm} C^2, &\hspace{-.1cm} C, &\hspace{-.1cm} C\varphi 
                                               \end{array}\mid \frac{-x^2}{4(1-x)}\right)\notag\\
&+ \frac{(q-1)\varphi(-1)\overline{C}(4)\overline{A}(1-x)g(\overline{A}C)g(\overline{B}C\varphi)}{g(\varphi)g(\overline{A}C^2)g(\overline{B})}\left[\frac{q-1}{q} {_{2}}F_1\left(\begin{array}{ccccccc}
	A, & B \\
	& C^2
\end{array}\mid \frac{-x^2}{4(1-x)} \right)\right.\notag\\
&\left.\times \delta(\overline{A}C)\delta(\overline{B}C\varphi)-{_{3}}F_2\left(\begin{array}{ccccccc}
               A, & B, & \overline{B}C^2 \\
                  & C^2, & C\varphi\end{array}\mid \frac{-x^2}{4(1-x)} \right)\delta(\overline{A}C)\right.\notag\\
                  &\left.- {{_3}}F_2\left(\begin{array}{ccccccc}
 	A, & B, & \overline{A}C^2 \\
 	& C^2, & C
 \end{array}\mid \frac{-x^2}{4(1-x)} \right)\delta(\overline{B}C\varphi)\right]+I_1 + I_2-I_3.
\end{align}
Finally combining \eqref{mt43eq14}, \eqref{mt43eq4}, \eqref{mt43eq6} and \eqref{mt43eq19}, we complete the proof. 
 \end{proof}
 \section{Values of Gaussian Hypergeometric Series}
 In this section, we will deduce the special values of Gaussian hypergeometric series. We first state two results of Greene on special values of Gaussian hypergeometric series.
 \begin{lemma}\emph{(\cite[(4.11)]{greene})}\label{g16}
   Let $A, B \in \widehat{\mathbb{F}_q^{\times}}$. Then we have
 \begin{align}
  {_{2}}F_1 \left(\begin{array}{ccc}
                        A, & B \\
                           & \overline{A}B 
                       \end{array}\mid -1\right)= \left\{
                                                 \begin{array}{ll}
                                                 0, & \hbox{if $B\neq\square$ ;} \\
                                                 {C\choose A} + {\varphi C\choose A}, & \hbox{if $B =C^{2}$.}
                                                 \end{array}
                                                 \right.\notag
 \end{align}
 \end{lemma}
 \begin{lemma}\emph{(\cite[(4.14)]{greene})}\label{g18}
 Let $A, B \in \widehat{\mathbb{F}_q^{\times}}$. Then we have
  \begin{align}
   {_{2}}F_1 \left(\begin{array}{ccc}
                        A, & B \\
                           & A^2 
                       \end{array}\mid 2\right)
 =A(-1)\left\{
       \begin{array}{ll}
       0, & \hbox{if $B\neq\square$ ;} \\
     {C\choose A} + {\varphi C\choose A}, & \hbox{if $B =C^{2}$.}
     \end{array}
     \right.\notag                      
 \end{align}
 \end{lemma}
 \begin{proof}[Proof of Theorem \ref{Value-41}] Putting $B=A\chi_4, C= A^4$  in Theorem \ref{MT41} and then using Lemma \ref{g1} we have
   \begin{align}\label{Value41eq-1}
&{_{4}}F_3\left(\begin{array}{cccccc}
            A^2, & A^2\varphi, & A^2\chi_4, & A^2\overline{\chi_4} \\
                 & A^4\varphi, & A^4, & \varphi 
            \end{array}\mid 4x(1-x) \right)\notag\\
&= \frac{\overline{A^2\chi_4}(4)g(A^2)g(\varphi)g(\overline{A^2}\varphi)^2}{qg(\overline{A^2})g(A^2\chi_4)g(\overline{A^2}\chi_4)}
{_{2}}F_1\left(\begin{array}{cccc}
                             A^2, & A^2\varphi \\
                                & A^4
                         \end{array}\mid x \right)
 {_{2}}F_1\left(\begin{array}{cccc}
         A^2, & A^2\varphi \\
             & \varphi
     \end{array}\mid x \right)\notag\\
&-\frac{A^2\varphi(x)\overline{A^2\chi_4}(4)\overline{A^6}(1-x)g(\varphi)}{qg(A^2\chi_4)g(\overline{A^2}\chi_4)}\delta\left(\frac{1-2x}{(1-x)^2}\right).
   \end{align}
Using Theorem \ref{thm7} (i) we have
\begin{align}\label{Value41eq-2}
&{_{2}}F_1\left(\begin{array}{cccc}
	A^2, & A^2\varphi \\
	& A^4
\end{array}\mid x \right)
{_{2}}F_1\left(\begin{array}{cccc}
	A^2, & A^2\varphi \\
	& \varphi
\end{array}\mid x \right)\notag\\
&={_{2}}F_1\left(\begin{array}{cccc}
A^2, & A^2\varphi \\
& \varphi
\end{array}\mid (1-x) \right)
{_{2}}F_1\left(\begin{array}{cccc}
A^2, & A^2\varphi \\
& \varphi
\end{array}\mid x \right)\notag\\
&= \frac{J(A^2\varphi, \overline{A^2})^2}{q^2}\left(\frac{1+\varphi(1-x)}{2}\right)\notag\\
&\times\left(\frac{1+\varphi(x)}{2}\right)
\left(\overline{A^4}(1+\sqrt{1-x})+ \overline{A^4}(1-\sqrt{1-x})\right)\left(\overline{A^4}(1+\sqrt{x})+ \overline{A^4}(1-\sqrt{x})\right).
\end{align}
 The last equality obtained by using \eqref{relation-Fu-Greene} and  Lemma \ref{33}. Finally, using \eqref{Value41eq-2}, \eqref{Value41eq-1}, 
 Lemma \ref{gj1} and Lemma \ref{g1} we complete the proof of (i).
 Replacing $x$ by $\frac{x}{x-1}$ in Theorem \ref{Value-41} (i), we complete the proof of (ii).
 \end{proof}
 \begin{proof}[Proof of Theorem \ref{Value-44}]
  Putting $x=-1$, $B= A^3\varphi$ in Corollary \ref{MT41C1} and then  using Lemma \ref{g1} on $g(A^2)g(\overline{A^2})$ and $g(A^4)g(\overline{A^4})$  we have
  \begin{align}\label{Value44eq1}
&{_{3}}F_2\left(\begin{array}{ccccccc}
              A^2, & A^6, & A^4\varphi \\
                   & A^8, & A^4
        \end{array}\mid -8\right)= \frac{\overline{A}(256)g(A^2)^2g(\overline{A^6})}{qg(\overline{A^2})}{_{2}}F_1\left(\begin{array}{ccccccc}
                                                A^2, & A^6 \\
                                                     & A^4 
                                              \end{array}\mid -1\right)^2 \notag\\
&\hspace{2cm}- \frac{\overline{A}(4096)}{q}- \frac{(q-1)}{q^3}\varphi(2)\overline{A}(4096)g(\overline{A^2}\varphi)g(A^2\varphi)\delta(A^4\varphi).    \end{align}
We complete the proof by combining Lemma \ref{g16} and \eqref{Value44eq1}.
 \end{proof}
 \begin{proof}[Proof of Theorem \ref{Value-45}]
  Putting $A=\chi_4$ and  $x=\frac{1+\sqrt{2}}{2}$ in Corollary \ref{MT41C1} and then using Lemma \ref{g1} we have 
  \begin{align}\label{Value45eq1}
  &{_{3}}F_2\left(\begin{array}{ccccccc}
              \varphi, & B^2, & B\chi_4 \\
                   & B^2\varphi, & B\overline{\chi_4} 
         \end{array}\mid -1\right)\notag\\
         &= \frac{\overline{B\chi_4}(4)g(\overline{B^2})g(B\chi_4)^2}{qg(\varphi)}{_{2}}F_1 \left(\begin{array}{ccc}
                 \varphi, & B^2 \\
                 & \overline{\chi_4}B   
                 \end{array}\mid \frac{1+\sqrt{2}}{2}\right)^2-\frac{B\chi_4(-1)}{q},            
\end{align}
where $B\neq \varepsilon, \varphi, \chi_4, \overline{\chi_4}$. From \cite[Thm. 1.11]{TB}, we have 
  \begin{align}\label{g20}
 &{_{2}}F_1 \left(\begin{array}{ccc}
                        \varphi, & B^2 \\
                             & \overline{\chi_4}B   
                         \end{array}\mid \frac{1\pm\sqrt{2}}{2}\right)\notag\\
 &=\frac{\chi_4(4)B(-4)g(\overline{\chi_4B})g(\chi_4)g(B\varphi)}{qg(\varphi)}
 \left\{
 \begin{array}{ll}
 0, & \hbox{if $B\neq\square$ ;} \\
 \displaystyle {D\choose \chi_4} + {D\varphi \choose \chi_4}, & \hbox{if $B= D^2$}.
 \end{array}
 \right.
 \end{align}
 Since $q\equiv 1 \pmod{8}$, we have $\chi_4(4)=\varphi(2)=1$. Now, combining \eqref{g20}, \eqref{Value45eq1}, and Lemma \ref{g1} we find that 
 \begin{align}\label{Value45eq2}
 	M&:= {_{3}}F_2\left(\begin{array}{ccccccc}
 		\varphi, & B^2, & B\chi_4 \\
 		& B^2\varphi, & B\overline{\chi_4} 
 	\end{array}\mid -1\right)\notag\\
 	&= -\frac{B(-1)}{q}+\frac{B(4)}{q^2g(\varphi)}g(\chi_4)^2g(B\varphi)^2g(\overline{B^2})
 	\left\{
 	\begin{array}{ll}
 		0, & \hbox{if $B\neq\square$ ;} \\
 		\displaystyle\left[{D\choose \chi_4} + {D\varphi \choose \chi_4}\right]^2, & \hbox{if $B = D^{2}$.}
 	\end{array}
 	\right.   
 \end{align}
 When $B= D^2$ we have 
  \begin{align}
 	M&= -\frac{1}{q}+\frac{D^2(4)}{q^2g(\varphi)}g(\chi_4)^2g(D^2\varphi)^2g(\overline{D^4})\left[{D\choose \chi_4} + {D\varphi \choose \chi_4}\right]^2\notag\\
 	 &= -\frac{1}{q}+\frac{g(\chi_4)g(D\chi_4)g(D\overline{\chi_4})g(\overline{D})g(\overline{D}\varphi)}{q^2g(\overline{\chi_4})}\left[{D\choose \chi_4} + {D\varphi \choose \chi_4}\right]^2.\notag
 \end{align}
The last equality is obtained by using Lemma \ref{g10} on $g(D^2\varphi)$ and Lemma \ref{HD2} on $g(\overline{D^4})$. Employing Lemma \ref{g8} and Lemma \ref{g1} we find that
 \begin{align}\label{Value45eq3}
	M&= \frac{D(-1)g(D)g(\overline{D}\varphi)g(D\chi_4)}{q^2g(D\overline{\chi_4})} + \frac{D(-1)g(\overline{D})g(D\varphi)g(D\overline{\chi_4})}{q^2g(D\chi_4)} + \frac{1}{q}.
\end{align}
Now using Lemma \ref{gj1} and Lemma \ref{g1} we have 
\begin{align}\label{Value45eq4}
	J(D, \varphi)J(\overline{D}\chi_{4}, \varphi)= \frac{qg(D)g(\overline{D}\chi_{4})}{g(D\varphi)g(\overline{D\chi_4})}.
\end{align}
 Using $\overline{g(A)}= A(-1)g(\overline{A})$, \eqref{Value45eq4} yields
\begin{align}\label{Value45eq4e1}
 \overline{J(D, \varphi)J(\overline{D}\chi_{4}, \varphi)}= \frac{qg(\overline{D})g(D\overline{\chi_{4}})}{g(\overline{D}\varphi)g(D\chi_4)}.
\end{align}
Combining \eqref{Value45eq4e1}, \eqref{Value45eq4}, \eqref{Value45eq3} and Lemma \ref{g1} we find that
\begin{align}\label{Value45eq6}
	M= \frac{1}{q} + \frac{2}{q^2}Re(J(D, \varphi)J(\overline{D}\chi_{4}, \varphi)),
\end{align}
where $B=D^2$. By Lemma \ref{sq-2} we have $B(-1)=-1$ if $B$ is not a square. Now, using \eqref{Value45eq6} in \eqref{Value45eq2} we have
\begin{align}
	&{_{3}}F_2\left(\begin{array}{ccccccc}
		\varphi, & B^2, & B\chi_4 \\
		& B^2\varphi, & B\overline{\chi_4} 
	\end{array}\mid -1\right)= \left\{
	\begin{array}{ll}
		\displaystyle\frac{1}{q}, & \hbox{if $B\neq\square$ ;} \vspace{.12cm}\\
		\displaystyle\frac{1}{q} + \frac{2}{q^2}Re(J(D, \varphi)J(\overline{D}\chi_{4}, \varphi)), & \hbox{if $B = D^{2}$.}
	\end{array}
	\right.\notag   
\end{align}
Clearly, $B\neq \varepsilon, \varphi, \chi_4, \overline{\chi_4}$ if and only if $B\overline{\chi_4}\neq \varepsilon, \varphi, \chi_4, \overline{\chi_4}$. We complete the proof of the theorem by 
putting $B= C\chi_4$.
 \end{proof}
\begin{proof}[Proof of Theorem \ref{Value-46}]
 Let $S$ be a multiplicative character which is a square, and let its order be strictly greater than $4$. Putting  $A= \sqrt{\overline{S^3}}$, 
 $B=\sqrt{S}$ and  $x=\frac{2+\sqrt{3}}{4}$ in Corollary \ref{MT41C1} and then using Lemma \ref{g1} we have
 \begin{align}\label{Value46eq1}
 &{_{3}}F_2\left(\begin{array}{ccccccc}
             \hspace{-.1cm} \overline{S^3}, &\hspace{-.2cm} S, &\hspace{-.2cm} \overline{S} \\
                   &\hspace{-.2cm} \overline{S^2}, &\hspace{-.2cm} \overline{S}\varphi 
         \end{array}\mid \frac{1}{4}\right)=\frac{S(4)g(\overline{S})g(S^2\varphi)^2}{qg(S^3)}
         {_{2}}F_1\left(\begin{array}{ccccccc}
             \hspace{-.1cm} \overline{S^3}, & \hspace{-.2cm}S \\
                   &\hspace{-.2cm} \overline{S}\varphi 
         \end{array}\mid \frac{2+\sqrt{3}}{4}\right)^2 -\frac{\overline{S}(4)}{q}.      
 \end{align}
From \cite[Thm. 1.10]{TB}, we have
 \begin{align}\label{g22}
&{_{2}}F_1 \left(\begin{array}{ccc}
                       \overline{S^3}, & \overline{S^2}\varphi \\
                            & \overline{S^4}  
                        \end{array}\mid \frac{4}{2\pm\sqrt{3}}\right)
=\frac{S^3(\sqrt{3})S(16)g(\overline{S^2}\varphi)g(\sqrt{S})}{S^3(\sqrt{3}\pm 2)g(\varphi)g(\sqrt{\overline{S^3}})}\notag\\
&\hspace{.5cm}\times\left\{
\begin{array}{ll}
0, & \hbox{if $q\equiv 11 \pmod{12}$;} \\
\displaystyle\frac{S(8)\overline{S}(27)J(\sqrt{\overline{S}},\sqrt{S^3}\varphi)}{J(\varphi,S)}\left[{S\choose \chi_3} + {S \choose \chi_3^2}\right], & \hbox{if $q\equiv  1 \pmod{12}.$}
\end{array}
\right.
\end{align}
Employing Lemma \ref{g21} into  \eqref{g22}, and then using Lemma \ref{g1}, we deduce from \eqref{Value46eq1} that
\begin{align}\label{Value46eq2}
	&M:= {_{3}}F_2\left(\begin{array}{ccccccc}
		\overline{S^3}, & S, & \overline{S} \\
		& \overline{S^2}, & \overline{S}\varphi 
	\end{array}\mid \frac{1}{4}\right)= -\frac{\overline{S}(4)}{q} +\frac{\varphi(-1)g(\overline{S})g(\sqrt{S})^2J(\sqrt{\overline{S}},\sqrt{S^3}\varphi)^2}{S(27)\overline{S}(16)g(S^3)g(\sqrt{\overline{S^3}})^2J(\varphi,S)^2}\notag\\
&\hspace{2cm}\times\left\{
\begin{array}{ll}
0, & \hbox{if $q\equiv 11 \pmod{12}$;} \\
\displaystyle\left[{S\choose \chi_3} + {S \choose \chi_3^2}\right]^2, & \hbox{if $q\equiv  1 \pmod{12}.$}
\end{array}
\right.       
\end{align}  
Using Lemma \ref{gj1} and Lemma \ref{g1} in \eqref{Value46eq2}, for $q\equiv  1 \pmod{12}$, we have
\begin{align}
	&M= -\frac{\overline{S}(4)}{q} +\frac{qS(16)\overline{S}(27)g(\overline{S})g(\sqrt{S^3}\varphi)^2}{g(S^3)g(\sqrt{\overline{S^3}})^2 g(S)^2}
	\left[{S\choose \chi_3} + {S \choose \chi_3^2}\right]^2.\notag  
\end{align}	  
Using Lemma \ref{g10} on $g(\sqrt{\overline{S^3}})$, and then employing Lemma \ref{g1} we find that 
\begin{align}
	&M= -\frac{\overline{S}(4)}{q} +\frac{q \overline{S}(4)\overline{S}(27)g(\overline{S})}{g(\overline{S^3})g(S)^2}
	\left[{S\choose \chi_3} + {S \choose \chi_3^2}\right]^2.\notag  
\end{align}	 
Lemma \ref{g8}, Lemma \ref{HD1}, and Lemma \ref{g1} yield
 \begin{align}\label{Value46eq3}
 	&M= \frac{\overline{S}(4)}{q^2}\left[q+ \frac{q g(\chi_3^2)^2}{g(\overline{S}\chi_3^2)g(S\chi_3^2)}
 	+ \frac{q g(\chi_3)^2}{g(\overline{S}\chi_3)g(S\chi_3)}\right]. 
 \end{align}
Using Lemma \ref{gj1} and Lemma \ref{g10} we have 
\begin{align}\label{Value46eq4}
	J(S, \chi_{3})J(\overline{S}, \chi_{3})= \frac{qg(\chi_{3})^2}{g(S\chi_3)g(\overline{S}\chi_3)}.
\end{align}
Using $\overline{g(A)}= A(-1)g(\overline{A})$, \eqref{Value46eq4} yields 
\begin{align}\label{Value46eq4e1}
 \overline{J(S, \chi_{3})J(\overline{S}, \chi_{3})}= \frac{qg(\chi_{3}^2)^2}{g(\overline{S}\chi_{3}^2)g(S\chi_{3}^2)}.
\end{align}
 Now, employing \eqref{Value46eq4e1}, \eqref{Value46eq4} and \eqref{Value46eq3} into \eqref{Value46eq2} we have
\begin{align}\label{Value46eq5}
	& {_{3}}F_2\left(\begin{array}{ccccccc}
		\overline{S^3}, & S, & \overline{S} \\
		& \overline{S^2}, & \overline{S}\varphi 
	\end{array}\mid \frac{1}{4}\right)\notag\\
	&= \left\{
	\begin{array}{ll}
		-\frac{\overline{S}(4)}{q}, & \hbox{if $q\equiv 11 \pmod{12}$;} \\
		\displaystyle\frac{\overline{S}(4)}{q}\left[q+ 2Re(J(S, \chi_{3})J(\overline{S}, \chi_{3}))\right], & \hbox{if $q\equiv  1 \pmod{12}.$}
	\end{array}
	\right.       
\end{align} 
Using \eqref{Greene-def-4}, Lemma \ref{g8}, and Lemma \ref{g1} we find that  
	\begin{align}\label{Value46eq6}
		& {_{3}}F_2\left(\begin{array}{ccccccc}
			\overline{S^3}, & S, & \overline{S} \\
			& \overline{S^2}, & \overline{S}\varphi 
		\end{array}\mid \frac{1}{4}\right)= {_{3}}F_2\left(\begin{array}{ccccccc}
			S, & \overline{S^3}, & \overline{S} \\
			& \overline{S^2}, & \overline{S}\varphi 
		\end{array}\mid \frac{1}{4}\right).
	\end{align}
Combining \eqref{Value46eq5} and \eqref{Value46eq6}, and then putting $S= \overline{C}$ we complete the proof of the theorem.
\end{proof}
\begin{remark}
 By Lemma \ref{sq-1} we have $\varphi(-1)=-1$ if $q\equiv 11\pmod{12}$ and $\varphi(-1)=1$ if $q\equiv 1\pmod{12}$. 
 To deduce Theorem \ref{Value-46} from \cite[Theorem 1.3]{EG-1}, we need to use the values of $\varphi(-1)$ accordingly.
\end{remark}
\begin{proof}[Proof of Theorem \ref{Value-43}]
  Putting $ B= \overline{A}, C= A^4$ and $x= 2$ in Theorem \ref{MT41} we have
  \begin{align}\label{Value43eq-1}
   &{_{2}}F_1\left(\begin{array}{cccc}
         A^2, & \overline{A^2} \\
             & A^4
             \end{array}\mid 2 \right) {_{2}}F_1\left(\begin{array}{cccc}
                                                       A^2, & \overline{A^2} \\
                                                             & \overline{A^4}
                                                                    \end{array}\mid 2 \right)\notag\\
 &= \frac{qg(\overline{A^2})g(\overline{A^4})g(A^4\varphi)}{g(A^2)^2g(\overline{A^6})g(\varphi)}{_{4}}F_3\left(\begin{array}{cccccc}
                     A^2, & \overline{A^2}, & \varepsilon, & \varphi \\
                          & \varepsilon, & A^4, & \overline{A^4} 
                     \end{array}\mid -8 \right)\notag\\
&-\frac{(q-1)g(\overline{A^2})g(\overline{A^4})g(A^4\varphi)}{g(A^2)^2g(\overline{A^6})g(\varphi)}{_{3}}F_2\left(\begin{array}{cccc}
         A^2, & \overline{A^2}, & \varphi\\
             & \varepsilon, & \overline{A^4}
             \end{array}\mid -8 \right)\delta(\overline{A^4}) \notag\\
             &+ \frac{(q-1)g(\overline{A^2})}{q^2g(A^2)^2g(\overline{A^6})}\left[\delta(\overline{A^4})+q\right].
 \end{align}            
From \eqref{Greene-def-4}, \eqref{b5} and using \eqref{b6} on ${A^2\chi\choose\chi}{\chi\choose A^4\chi}$ with the fact that $A^2\neq\varepsilon$, we obtain 
\begin{align}\label{Value43eq-3}
 {_{4}}F_3\left(\begin{array}{cccccc}
                    \hspace{-.1cm} A^2, & \hspace{-.2cm} \overline{A^2}, &\hspace{-.2cm} \varepsilon, &\hspace{-.2cm} \varphi \\
                          &\hspace{-.2cm} \varepsilon, &\hspace{-.2cm} A^4, &\hspace{-.2cm} \overline{A^4} 
                     \end{array}\mid -8 \right)&= \frac{q}{q-1} \sum_{\chi\in\widehat{\mathbb{F}_q^{\times}}}{A^2\chi\choose\chi}{\overline{A^2}\chi\choose\chi}{\chi\choose A^4\chi}{\varphi\chi\choose\overline{A^4}\chi}\chi(-8)\notag\\
                     &={\overline{A^2}\choose\overline{A^4}}{_{3}}F_2\left(\begin{array}{cccccc}
                     \hspace{-.1cm} \overline{A^2}, &\hspace{-.2cm} A^2 , &\hspace{-.2cm} \varphi \\
                           &\hspace{-.2cm} A^4, &\hspace{-.2cm} \overline{A^4} 
                     \end{array}\mid -8 \right)+\frac{1}{q^2}{\varphi\choose\overline{A^4}}.
\end{align}
Using \eqref{prop-300} and the fact that $A^2\neq\varepsilon$ we have 
\begin{align}\label{Value43eq-4}
 {_{2}}F_1\left(\begin{array}{cccc}
                A^2, & \overline{A^2} \\
                    & \overline{A^4}
              \end{array}\mid 2 \right)&= {\overline{A^2}\choose\overline{A^4}}{A^2\choose\overline{A^4}}^{-1}{_{2}}F_1\left(\begin{array}{cccc}
                \overline{A^2}, & A^2 \\
                    & \overline{A^4}
              \end{array}\mid 2 \right)\notag\\
              &=\frac{g(\overline{A^2})g(A^6)}{g(A^2)^2}{_{2}}F_1\left(\begin{array}{cccc}
                \overline{A^2}, & A^2 \\
                    & \overline{A^4}
              \end{array}\mid 2 \right).
\end{align}
The last equality is obtained by using Lemma \ref{g8}. Finally, employing \eqref{Value43eq-4}, \eqref{Value43eq-3}, Lemma \ref{g1} and Lemma \ref{g18} into \eqref{Value43eq-1} 
we complete the proof.
 \end{proof}
\begin{proof}[Proof of Theorem \ref{MTV43C1}]
Using Lemma \ref{g8} and Lemma \ref{g1} in Theorem \ref{Value-43} we have
\begin{align}\label{mtv43c1eq1}
&{_{3}}F_2\left(\begin{array}{ccccccc}
\overline{A^2}, & A^2 , & \varphi \\
& A^4, & \overline{A^4}
\end{array}\mid -8\right)
\notag\\
&= \frac{g(\varphi)g(A^2)}{qg(\overline{A^2})g(A^4\varphi)} + \frac{g(A)g(A^2)g(\varphi)g(\overline{A}\varphi)}{qg(\overline{A^2})g(A^3)g(\varphi\overline{A^3})g(A^4\varphi)}
+ \frac{g(\overline{A})g(A^2)g(A\varphi)g(\varphi)}{qg(\overline{A^2})g(\overline{A^3})g(\varphi A^3)g(A^4\varphi)}.
\end{align}
Now \eqref{Greene-def-4} and Lemma \ref{g8} yield
\begin{align}\label{mtv43c1eq2}
{_{3}}F_2\left(\begin{array}{ccccccc}
\overline{A^2}, & A^2 , & \varphi \\
& A^4, & \overline{A^4}
\end{array}\mid -8\right)= \frac{g(\varphi)g(A^2)}{g(\overline{A^2})g(A^4\varphi)}{_{3}}F_2\left(\begin{array}{ccccccc}
\varphi, & A^2 , & \overline{A^2} \\
& A^4, & \overline{A^4}
\end{array}\mid -8\right).
\end{align}
Combining \eqref{mtv43c1eq2} and \eqref{mtv43c1eq1} we have
\begin{align}\label{mtv43c1eq3}
{_{3}}F_2\left(\begin{array}{ccccccc}
\varphi, & A^2 , & \overline{A^2} \\
& A^4, & \overline{A^4}
\end{array}\mid -8\right)= \frac{1}{q} + \frac{g(A)g(\overline{A}\varphi)}{qg(A^3)g(\varphi\overline{A^3})}
+ \frac{g(\overline{A})g(A\varphi)}{qg(\overline{A^3})g(\varphi A^3)}.
\end{align}
Using Lemma  \ref{gj1} we find that 
\begin{align}\label{mtv43c1eq4}
&\frac{\overline{A^2}(4)J(\overline{A^2}, A^6)}{q^2J(A^2, A^2)}\left[J(A^2, A)^2 + J(A^2, A\varphi)^2\right]\notag\\
&= \frac{\overline{A^2}(4)}{q^2}\left[\frac{g(A)^2g(A^6)g(\overline{A^2})}{g(A^3)^2} + \frac{g(A\varphi)^2g(A^6)g(\overline{A^2})}{g(A^3\varphi)^2}\right]\notag\\
&= \frac{g(A)g(\overline{A}\varphi)}{qg(A^3)g(\varphi\overline{A^3})}
+ \frac{g(\overline{A})g(A\varphi)}{qg(\overline{A^3})g(\varphi A^3)}.
\end{align}
The last equality is obtained by using Lemma \ref{g10} on $g(\overline{A^2})$ and $g(A^6)$, and then we use Lemma \ref{g1}. 
We now complete the proof by combining \eqref{mtv43c1eq4} and \eqref{mtv43c1eq3}.
\end{proof}
\begin{proof}[Proof of Theorem \ref{Value-49}] 
Given that $S$ is a character which is a square and its order is strictly greater than $4$. Putting	$A= \overline{S^3}$, $B= \overline{S^2}\varphi$, $C= \overline{S^2}$ 
and $x= \frac{4}{2+\sqrt{3}}$ in Theorem \ref{MT43} and using the fact that $g(\varepsilon)= -1$ we have
\begin{align}\label{Value49eq1}
	&{_{2}}F_1\left(\begin{array}{ccccccc}
		\overline{S^3}, & \overline{S^2}\varphi \\
		& \overline{S^4}
	\end{array}\mid \frac{4}{2+\sqrt{3}}\right)^2
= \frac{(q-1)S(256)S(\sqrt{3}-2)\overline{S^5}(2+\sqrt{3})}{g(\overline{S^3})g(S^2\varphi)^2g(\overline{S})}\notag\\
&+\frac{(q-1)\varphi(-1)S(16)S^3(\sqrt{3}-2)\overline{S^3}(\sqrt{3}+2)g(S)}{g(\varphi)g(\overline{S})g(S^2\varphi)}{_{3}}F_2\left(\begin{array}{ccccccc}
		\overline{S^3}, & \overline{S^2}\varphi, & \overline{S} \\
		& \overline{S^4}, & \overline{S^2}
	\end{array}\mid 4\right)\notag\\
&-\frac{q\varphi(-1)S(16)S^3(\sqrt{3}-2)\overline{S^3}(\sqrt{3}+2)g(S)}{g(\varphi)g(\overline{S})g(S^2\varphi)}{_{4}}F_3\left(\begin{array}{ccccccc}
		\overline{S^3}, & \overline{S^2}\varphi, & \overline{S}, & \overline{S^2}\varphi \\
		& \overline{S^4}, & \overline{S^2}, & \overline{S^2}\varphi 
	\end{array}\mid 4\right).
\end{align}  
Now using \eqref{Greene-def-4} and \eqref{b5} we have 
\begin{align}\label{Value49eq2}
{_{4}}F_3\left(\begin{array}{ccccccc}
	\overline{S^3}, & \overline{S^2}\varphi, & \overline{S}, & \overline{S^2}\varphi \\
	& \overline{S^4}, & \overline{S^2}, & \overline{S^2}\varphi 
\end{array}\mid 4\right)
= -\frac{1}{q}{_{3}}F_2\left(\begin{array}{ccccccc}
\overline{S^3}, & \overline{S^2}\varphi, & \overline{S} \\
& \overline{S^4}, & \overline{S^2}
\end{array}\mid 4\right) + I_1,
\end{align}
where
\begin{align}\label{Value49eq3}
I_1&=\sum_{\chi}{\overline{S^3}\chi\choose \chi}{\overline{S^2}\varphi\chi\choose \overline{S^4}\chi}{\overline{S}\chi\choose \overline{S^2}\chi}\chi(4)\delta(\overline{S^2}\varphi\chi)\notag\\
&=-\frac{\varphi(-1)S(16)}{q}{\overline{S}\varphi\choose S^2\varphi}{S\varphi\choose \varphi}= -\frac{S(16)g(\varphi)g(\overline{S^2}\varphi)}{q^2g(S)g(\overline{S^3})}
= -\frac{\varphi(-1)S(16)g(\varphi)}{qg(S)g(\overline{S^3})g(S^2\varphi)}.
\end{align}
The above equality is obtained by putting $\chi= S^2\varphi$ and then using \eqref{b7}, Lemma \ref{g8}, and Lemma \ref{g1}. 
By combining \eqref{Value49eq1}, \eqref{g22}, \eqref{Value49eq2},  \eqref{Value49eq3} and Lemma \ref{g1}, we have
\begin{align}\label{Value49eq4}
&	{_{3}}F_2\left(\begin{array}{ccccccc}
		\overline{S^3}, & \overline{S^2}\varphi, & \overline{S} \\
		& \overline{S^4}, & \overline{S^2} 
	\end{array}\mid 4\right)\notag\\
	&= -\frac{\varphi(-1)S(16)g(\varphi)}{g(\overline{S^3})g(S^2\varphi)g(S)}+\frac{\varphi(-1)g(\overline{S})g(\varphi)g(\overline{S^2}\varphi)g(\sqrt{S})^2J(\sqrt{\overline{S}}, \sqrt{S^3}\varphi)^2}{q\overline{S}(1024)S(27)g(S)
	g(\sqrt{\overline{S^3}})^2J(\varphi, S)^2}\notag\\
	&\times\left\{
	\begin{array}{ll}
		0, & \hbox{if $q \equiv 11 \mod(12) $;} \\
		\displaystyle  \left[{S\choose \chi_3} + {S \choose \chi_3^2}\right]^2, & \hbox{if $q \equiv 1 \mod(12) $.}
	\end{array}\right.
\end{align}
Lemma \ref{gj1} and Lemma \ref{g1} yield
\begin{align}\label{Value49eq5}
	& \frac{\varphi(-1)S(1024)\overline{S}(27)g(\overline{S})g(\varphi)g(\overline{S^2}\varphi)g(\sqrt{S})^2J(\sqrt{\overline{S}}, \sqrt{S^3}\varphi)^2}{qg(S)g(\sqrt{\overline{S^3}})^2J(\varphi, S)^2}\notag\\
	&= \frac{q\varphi(-1)S(1024)\overline{S}(27)g(\overline{S})g(\overline{S^2}\varphi)g(\sqrt{S^3}\varphi)^2}{g(S)^3g(\sqrt{\overline{S^3}})^2g(\varphi)}\notag\\
	&= \frac{\varphi(-1)S(16)\overline{S}(27)g(\overline{S})g(\overline{S^2}\varphi)g(S^3)^2g(\varphi)}{qg(S)^3}.
\end{align}
The last equality is obtained by using Lemma \ref{g10} on $g(\sqrt{S^3}\varphi)$ and then using Lemma \ref{g1}. Combining \eqref{Value49eq5} and \eqref{Value49eq4} we have
\begin{align}\label{Value49eq6}
	&{_{3}}F_2\left(\begin{array}{ccccccc}
	\overline{S^3}, & \overline{S^2}\varphi, & \overline{S} \\
		& \overline{S^4}, & \overline{S^2} 
	\end{array}\mid 4\right)= -\frac{\varphi(-1)S(16)g(\varphi)}{g(\overline{S^3})g(S^2\varphi)g(S)}
	+ \frac{\varphi(-1)g(\overline{S})g(\overline{S^2}\varphi)g(S^3)^2g(\varphi)}{qS(27)\overline{S}(16) g(S)^3}\notag\\
	&\hspace{2cm}\times\left\{
	\begin{array}{ll}
		0, & \hbox{if $q \equiv 11 \mod(12) $;} \\
		\displaystyle  \left[{S\choose \chi_3} + {S \choose \chi_3^2}\right]^2, & \hbox{if $q \equiv 1 \mod(12) $.}
	\end{array}\right.
\end{align} 
Lemma \ref{g8} and \eqref{Greene-def-4} yield 
\begin{align}\label{Value49eq7}
	&{_{3}}F_2\left(\begin{array}{ccccccc}
		\overline{S^3}, & \overline{S^2}\varphi, & \overline{S} \\
		& \overline{S^4}, & \overline{S^2} 
	\end{array}\mid 4\right)= \frac{g(S^3)g(\varphi)}{g(S^2\varphi)g(S)}{_{3}}F_2\left(\begin{array}{ccccccc}
		\overline{S^3}, & \overline{S}, & \overline{S^2}\varphi \\
		& \overline{S^4}, & \overline{S^2} 
	\end{array}\mid 4\right).
\end{align} 
Using \eqref{Value49eq7} in \eqref{Value49eq6} and then using Lemma \ref{g1} we have
\begin{align}\label{Value49e1eq7}
	&{_{3}}F_2\left(\begin{array}{ccccccc}
		\overline{S^3}, & \overline{S}, &  \overline{S^2}\varphi \\
		& \overline{S^4}, & \overline{S^2} 
	\end{array}\mid 4\right)=-\frac{\varphi(-1)S(16)}{q}\notag\\
	&+ \frac{S(16)\overline{S}(27)g(\overline{S})g(S^3)}{g(S)^2}\left\{
	\begin{array}{ll}
		0, & \hbox{if $q \equiv 11 \mod(12) $;} \\
		\displaystyle  \left[{S\choose \chi_3} + {S \choose \chi_3^2}\right]^2, & \hbox{if $q \equiv 1 \mod(12)$.}
	\end{array}\right.
\end{align} 
Using Lemma \ref{gj1} and Lemma \ref{g1} we have 
\begin{align}\label{Value49e1eq8}
	\frac{J(\overline{S}, \overline{S})}{J(\overline{S^3}, S)}= \frac{g(\overline{S})g(S^3)}{g(S)^2}.
\end{align}
Finally, combining \eqref{Value49e1eq8} and \eqref{Value49e1eq7} we complete the proof.
\end{proof}	
	

\end{document}